\documentclass[a4paper,10pt,leqno]{amsart}

\usepackage{amsmath,amsfonts,amsthm,amssymb,dsfont}
\usepackage[alphabetic, abbrev]{amsrefs}
\usepackage[OT4]{fontenc}
\usepackage[shortlabels]{enumitem}
\usepackage{mathrsfs}
\usepackage{xspace}
\usepackage[pdftex]{graphicx}
\usepackage{color}

\long\def\symbolfootnote[#1]#2{\begingroup
\def\thefootnote{\fnsymbol{footnote}}\footnote[#1]{#2}\endgroup}

\newtheorem{maintheorem}{Theorem}

\newtheorem{maincorollary}[maintheorem]{Corollary}

\newtheorem{theorem}{Theorem}[section]
\newtheorem{lemma}[theorem]{Lemma}

\newtheorem{prop}[theorem]{Proposition}
\newtheorem{cor}[theorem]{Corollary}

\theoremstyle{definition}
\newtheorem{rem}[theorem]{Remark}
\newtheorem{defin}[theorem]{Definition}

\DeclareMathOperator{\CAT}{CAT}
\DeclareMathOperator{\Tame}{Tame}
\DeclareMathOperator{\STame}{STame}
\DeclareMathOperator{\Aut}{Aut}
\DeclareMathOperator{\SAut}{SAut}
\DeclareMathOperator{\Bir}{Bir}
\DeclareMathOperator{\GL}{GL}
\DeclareMathOperator{\Stab}{Stab}
\DeclareMathOperator{\Jac}{Jac}
\DeclareMathOperator{\inn}{in}
\DeclareMathOperator{\out}{out}
\DeclareMathOperator{\non}{non}

\newcommand{\A}{{\mathbf{k}}}
\newcommand{\TA}{\Tame(\A^3)}
\newcommand{\STA}{\STame(\A^3)}
\newcommand{\K}{\mathbf{k}}
\newcommand{\Comp}{\mathcal{C}}
\newcommand{\C}{\mathbb{C}}
\newcommand{\Z}{\mathbb{Z}}
\newcommand{\R}{\mathbb{R}}
\newcommand{\LL}{\mathcal{L}}
\newcommand{\T}{\mathcal{T}}
\newcommand{\D}{\mathcal{D}}
\newcommand{\p}{\mathbb{P}}

\renewcommand{\le}{\leqslant}
\renewcommand{\ge}{\geqslant}
\renewcommand{\leq}{\leqslant}
\renewcommand{\geq}{\geqslant}

\usepackage[cmtip]{xy}
\xyoption{pdf}
\xyoption{color}
\xyoption{all}
\newcommand{\mygraph}[1]{\xybox{\xygraph{#1}}}
\newcommand{\typeone}{\xy*\cir<2pt>{}\endxy}
\newcommand{\typetwo}{\bullet}
\newcommand{\typethree}{\rule[.2ex]{0.9ex}{0.9ex}}

\newcommand{\mycolor}{blue}
\usepackage[pdfauthor={S. Lamy \& P. Przytycki}, colorlinks, linktocpage, citecolor = \mycolor,
linkcolor = \mycolor, urlcolor = \mycolor]{hyperref}
\usepackage[all]{hypcap} 

\begin{document}

\title[Acylindrical hyperbolicity of Tame$(\A^3)$]
{Acylindrical hyperbolicity of the three-dimensional tame automorphism group}

\author[S.~Lamy]{St\'ephane Lamy$^{\dag}$}
\thanks{$^{\dag}$ This research was partially supported by ANR Grant ``BirPol''
ANR-11-JS01-004-01.}
\address{Institut de Math\'ematiques de Toulouse, Universit\'e Paul Sabatier,
118 route de Narbonne, 31062 Toulouse Cedex 9, France}
\email{slamy@math.univ-toulouse.fr}

\author[P.~Przytycki]{Piotr Przytycki$\ddag$}
\address{Dept. of Math. \& Stats., McGill University, Montreal, Quebec, Canada H3A 0B9}
\email{piotr.przytycki@mcgill.ca}
\thanks{$\ddag$ Partially supported by NSERC, FRQNT, National Science Centre DEC-2012/06/A/ST1/00259, and UMO-2015/\-18/\-M/\-ST1/\-00050.}

\maketitle

\begin{abstract}
\noindent
We prove that the group $\STA$ of special tame automorphisms  of
the affine 3-space is not simple, over any base field of characteristic zero.
Our proof is based on the study of the geometry of a 2-dimensional
simply-connected simplicial complex $\Comp$ on which the tame automorphism
group acts naturally.
We prove that $\Comp$ is contractible and Gromov-hyperbolic, and we prove that
$\TA$ is acylindrically hyperbolic by finding explicit loxodromic weakly
proper discontinuous elements.
\end{abstract}

\section{Introduction}
\label{sec:intro}

The \emph{tame automorphism group} of the affine space $\A^3$, over a base field
$\K$ of characteristic zero, is the subgroup of the polynomial automorphism
group $\Aut(\A^3)$ generated by the affine and elementary automorphisms:
$$\TA = \langle A_3, E_3 \rangle,$$
where
\begin{align*}
A_3 &= \GL_3(\K) \ltimes \K^3 , \ \mathrm{and}\\
E_3 &= \{ (x_1, x_2, x_3) \mapsto (x_1 + P(x_2, x_3), x_2, x_3)\ | \ P \in
\K[x_2,x_3] \}.
\end{align*}
There is a natural homomorphism $\Jac\colon \TA \to \K^*$ given by the
Jacobian determinant. The kernel $\STA$ of
this homomorphism is the \emph{special tame automorphism group}.
Analogously one defines $\Tame(\A^n)$ and $\STame(\A^n)$ for arbitrary $n\geq 2$.
It is a natural question whether $\STame(\A^n)$ is a simple group.
In this paper we prove that $\STA$ is not simple (and indeed very far from being simple).

Our strategy is to use an action of $\TA$ on a Gromov-hyperbolic triangle
complex, and to exhibit a loxodromic weakly proper discontinuous element of $\STA$, in the sense of M. Bestvina and K. Fujiwara \cite{BF}.
Recall that an isometry $f$ of a metric space $X$ is \emph{loxodromic} if for some (hence any) $x\in X$
there exists $\lambda > 0$ such that for any $k\in \Z$ we have $|x,f^k\cdot x|\geq \lambda |k|$. Suppose that $f$ belongs to a group
$G$ acting on $X$ by isometries. We say that $f$ is
\emph{weakly proper discontinuous} (WPD) if for some (hence any) $x\in X$ and any $C\geq 1$, for $k$ sufficiently large
there are only finitely many $j\in G$ satisfying $|x, j\cdot x| \leq C$ and $|f^k\cdot x, j\circ f^k\cdot x|\leq C$.

By the work of F. Dahmani, V. Guirardel,
and D. Osin \cite[Thm 8.7]{DGO}, the existence of an action of a non-virtually cyclic group $G$ on a
Gromov-hyperbolic metric space, with at least one loxodromic WPD element, implies that $G$
has a free normal subgroup, and in particular $G$ is not simple.
By the work of D. Osin \cite[Thm 1.2]{Osin}, such a group is \emph{acylindrically hyperbolic}:
there exists a (different) Gromov-hyperbolic space on which the action of $G$ is acylindrical, a notion introduced for general metric spaces by B. Bowditch \cite{B}.

This strategy was recently applied to various transformation groups in algebraic geometry.
We now review a few examples to explain how the group $\TA$ fits in the global picture.

First we discuss the group $\Bir(\p^2_\K)$, the \emph{Cremona group} of rank 2, which is the group of birational transformations of the projective plane.
It is by no means obvious to find a Gromov-hyperbolic space on which the Cremona groups acts.
One takes all projective surfaces dominating $\p^2_\K$ by a sequence of blow-ups, and considers the direct limit of their spaces of curves, called N\'eron--Severi groups. The limit is endowed with a lorentzian intersection form defining an infinite dimensional hyperboloid.
This hyperboloid was introduced in \cite{Cantat} and used to prove for instance a Tits alternative for the Cremona group.
Then it was used in \cite{CL} to prove the non-simplicity of $\Bir(\p^2_\K)$ over an algebraically closed field $\K$. Finally, the above mentioned strategy was applied in \cite{Lonjou} to obtain the non-simplicity over an arbitrary base field.

Note that one of the initial motivations for \cite{DGO} was the application to the mapping class group.
As it is the case for the Cremona group, in studying the mapping class group one uses an action on a non-locally compact Gromov-hyperbolic space (the complex of curves), but the parallel goes beyond that. For instance, there are striking similarities between the notion of dilatation factor for a pseudo-Anosov map, and the dynamical degree of a generic Cremona map: see the survey \cite{Cantat-survey} for more details.

The above results about $\Bir(\p^2_\K)$ were inspired by previous work on its subgroup $\Aut(\K^2)=\Tame(\A^2)$. It is classical that $\Aut(\K^2)$ is the amalgamated product of two of its subgroups, and so we get an action of $\Aut(\K^2)$ on the associated Bass--Serre tree (which is obviously Gromov-hyperbolic).
Together with some classical small cancellation theory this was used by V. Danilov \cite{Danilov} to produce many normal subgroups in $\Aut(\C^2)$ (and in $\SAut(\C^2)$), see also \cite{FL}.
Recently these results were extended to the case of an arbitrary field by A. Minasyan and D. Osin \cite[Cor~ 2.7]{MO}, again by producing concrete examples of WPD elements.

When one tries to extend these results to higher dimensions one has to face the formidable gap in complexity between birational geometry of surfaces and in higher dimension.
We refer to the introduction of \cite{BFL} for a few more comments on this side of the story.
The group $\TA$ seems to be a good first step to enter the world of dimension 3.
It was a classical question proposed by M. Nagata in the 70' whether the inclusion $\TA \subset \Aut(\K^3)$ was strict. This was confirmed 30 years later by I.P. Shestakov and U.U. Umirbaev \cite{SU}, with an argument which was somewhat simplified by S. Kuroda \cite{Kuroda}. Then it was recently noticed \cite{BFL, wright, Lamy} that we can rephrase the theory developped by these authors by saying that $\TA$ is the amalgamated product of three subgroups along their pairwise intersections. Equivalently, the group $\TA$ acts on a simply connected 2-dimensional simplicial complex $\Comp$, with fundamental domain a single triangle.
This complex $\Comp$ is the main object of our present work
(see Section~\ref{sec:complex} for the definition).

To end the historical background, note that the situation on the affine 3-dimensional quadric has been also successfully explored.
(In fact, we considered it a test setting for the whole strategy, before trying to handle the affine space $\K^3$.) The notion of a tame automorphism in this context was introduced by S. V\'en\'ereau and the first author. An action on a Gromov-hyperbolic $\CAT(0)$ square complex was constructed in \cite{BFL}. WPD elements were recently produced by A. Martin \cite{Martin}.

\medskip

Our first two results about the geometry of the complex $\Comp$ associated with $\TA$ are the following.

\begin{maintheorem}
\label{thm:contractible}
$\Comp$ is contractible.
\end{maintheorem}

\begin{maintheorem}
\label{thm:hyperbolic}
$\Comp$ is Gromov-hyperbolic.
\end{maintheorem}

As is often the case when dealing with 2-dimensional complexes, our arguments rely on understanding disc diagrams, i.e.\ simplicial discs mapping to the complex.
The main difficulty here is that, by contrast with the above mentioned settings, the complex $\Comp$ does not admit an equivariant $\CAT(0)$ metric.
We circumvent this problem by a procedure of ``transport of curvature''.
Precisely, given a disc diagram, first we assign to each triangle of the diagram angles $\pi/2, \pi/3, \pi/6$, and then, by putting an orientation on certain edges of the diagram, we describe how to transport any excess of positive curvature at a given vertex to neighbouring vertices.
In this sense we obtain that any disc diagram is negatively curved, which gives Theorem~\ref{thm:contractible}, and also, via linear isoperimetric inequality, Theorem~\ref{thm:hyperbolic}.

\medskip

We now turn to the existence of WPD elements in $\TA$.
On the $1$-skeleton $\Comp^1$ of $\Comp$ we use the path metric where each edge
has length $1$, which is
quasi-isometric to any $\TA$-equivariant path-metric on $\Comp$.
Let $n \ge 0$, and let $g,h,f\in\TA$ be the automorphisms defined by
\begin{align*}
g^{-1}(x_1,x_2,x_3) & = (x_2, x_1+x_2x_3, x_3), \\
h^{-1}(x_1,x_2,x_3) &= (x_3,x_1,x_2),\\
f &= g^n \circ h.
\end{align*}

Observe that  $\Jac(h) =
1$ and $\Jac(g) = -1$, so for even $n$ the automorphism $f$ is an
element
of $\STA$.

\begin{maintheorem}
\label{thm:loxo}
Let $n \ge 3$. Then $f$ acts as a loxodromic element on $\Comp^1$.
In particular, the complex $\Comp$ has infinite diameter.
\end{maintheorem}

\begin{maintheorem}
\label{thm:wpd}
Let $n\geq12$ be even and $G=\TA$. Then $f$ acts as a $\mathrm{WPD}$ element on~$\Comp^1$.
\end{maintheorem}

Again the idea of proof relies on the notion of combinatorial curvature.
There is a simplicial embedding $\gamma\colon \R \to \Comp$ such that $f$ acts by translation on the image of $\gamma$.
Moreover, we prove that any disc diagram containing a segment of $\gamma$ in its boundary has very negative curvature along this segment.
This is the key property that allows to obtain both Theorems \ref{thm:loxo} and \ref{thm:wpd}.

As mentioned above, by \cite[Thm 8.7]{DGO}, Theorems~\ref{thm:loxo}
and~\ref{thm:wpd} immediately imply the following:

\begin{maincorollary}
The group $\STA$ is not simple.
\end{maincorollary}

Furthermore, by \cite[Thm 1.2]{Osin}, we have:

\begin{maincorollary}
The groups $\TA$ and $\STA$ are acylindrically hyperbolic.
\end{maincorollary}

\medskip
\noindent \textbf{Organisation.} We recall the definition of $\Comp$ in Section~\ref{sec:complex}. In Section~\ref{sec:trees} we study its local geometry.
This allows us in Section~\ref{sec:curvature} to introduce the \emph{curvature} at vertices of reduced disc diagrams in $\Comp$. In Section~\ref{sec:nonpositive} we prove that curvature is nonpositive and we deduce Theorem~\ref{thm:contractible}.
A more detailed study allows to prove Theorem~\ref{thm:hyperbolic} in Section~\ref{sec:hyperbolicity}.
Finally, in Section~\ref{sec:loxodromic} we prove Theorems~\ref{thm:loxo} and~\ref{thm:wpd}.

\section{Complex}
\label{sec:complex}

\begin{figure}[b]
$$
\mygraph{
!{<0cm,0cm>;<1.5cm,0cm>:<0cm,1cm>::}
!{(-2,0)}*{\typethree}="Q"
!{(0,2)}*-{\typethree}="id"
!{(0,-2)}*{\typethree}="v3''"
!{(2,0)}*{\typethree}="u3"
!{(0,0)}*{\typeone}="x2"
!{(-1,1)}*-{\typetwo}="x2x3"
!{(-1,-1)}*-{\typetwo}="v2''"
!{(1,1)}*-{\typetwo}="u2'"
!{(1,-1)}*-{\typetwo}="x1x2"
!{(-1.5,3)}*{\typeone}="x3"
!{(1.5,3)}*{\typeone}="x1"
!{(0,3.5)}*{\typetwo}="x1x3"
"id"-_<(.28){[x_1, x_2, x_3]}_>(1.1){[x_2,x_3]}"x2x3"-"Q"-_>{[x_1 + Q(x_2),x_2] \quad}"v2''"-"v3''"-"x1x2"-_<{\quad [x_2, x_3+P(x_1,x_2)]}"u3"-"u2'"_>(.85){[x_1,x_2]}-"id"
"x1x2"-"x2"-"v2''"
"v3''"-|<(-.15){[x_1 + Q(x_2),x_2,x_3+P(x_1,x_2)]}"x2"-"id"
"u3"-_>(.75){[x_2]}"x2"-"Q"
"x2x3"-"x2"-"u2'"
"x2x3"-"x3"-^<{[x_3]}^>(1.1){[x_1,x_3]}"x1x3"-^>{[x_1]}"x1"-"u2'"
"id"-"x3" "id"-"x1x3" "id"-"x1"
"Q"-^<{[x_1 + Q(x_2),x_2,x_3]}"x3" "u3"-_<{[x_1,x_2,x_3+P(x_1,x_2)]}"x1"
}
$$
\caption{Some triangles of the complex $\Comp$}\label{fig:complex}
\end{figure}

In this section we recall the construction of the simplicial complex $\Comp_n$ associated
with ${\Tame(\A^n)}$, for $n\geq 2$. See \cite[\S 2.5]{BFL}, \cite{Lamy} for
more details.
For any $1 \le r \le n$, an \emph{$r$-tuple of components} is a polynomial map
\begin{align*}
f\colon\A^n &\to \A^r \\
x = (x_1, \dots, x_n) &\mapsto \left( f_1(x), \dots, f_r(x) \right)
\end{align*}
that can be extended to a tame automorphism $f = (f_1,\dots, f_n)$ of $\A^n$.
We consider the orbits of the action of the affine automorphism group $A_r=\GL_r(\K) \ltimes \K^r$ on the $r$-tuples of components:
\begin{equation*}
[f_1, \dots,f_r] = A_r (f_1, \dots, f_r) = \{ a \circ (f_1, \dots, f_r) \ | \ a \in
A_r\}
\end{equation*}
A \emph{vertex of type $r$} of $\Comp_n$ is such an orbit $[f_1, \dots,f_r]$, and will be usually denoted as $v_r$.
The vertices  $[f_1], [f_1,f_2],\ldots, [f_1, \dots, f_r]$ span an $(r-1)$-simplex of $\Comp_n$.
The tame automorphism group acts on $\Comp_n$ by
isometries, via pre-composition:
$$g\cdot [f_1,\dots,f_r] = [f_1 \circ g^{-1},\dots, f_r \circ g^{-1}]$$
Note that this action is transitive on $(n-1)$-simplices, and in particular, it is transitive on the vertices of any type.
It is easy to see that $\Comp_n$ is connected.

We shall use $\Comp_n$ mainly in the case $n = 3$, but also in the
case $n = 2$, to study the link of a vertex in $\Comp_3$. In fact, $\Comp_2$ is the Bass--Serre tree
corresponding to the splitting $\Aut(\A^2)=A_2\ast_{A_2\cap E_2} E_2$, with $$E_2=\{ (x_1, x_2) \mapsto (ax_1 + P(x_2), bx_2+c)\ | \ a,b\in \K^*, c\in \K\}$$ (see \cite[Prop 2.16]{BFL}). We will also denote $\Comp_2$ as $\T_\K$ to emphasize
the field. Later also the field $\K(x_3)$ will be used instead of $\K$.
We shortly denote by $\Comp = \Comp_3$ the triangle complex for $\TA$.
It has vertices of type $1,2$ and $3$.
We say that an edge in $\Comp$ has \emph{type
$(i,j)$} if it joins a vertex of type $i$ with a vertex of type $j$. By
$\Comp^0, \Comp^1$ we denote the vertex set and the $1$-skeleton of $\Comp$.

We illustrate a part of $\Comp$ in Figure~\ref{fig:complex}. Vertices of types $1,2,$ and $3$ are represented by symbols $\typeone, \typetwo, \typethree,$ respectively.

\section{Trees and links} \label{sec:trees}

In this section we define and explain the relation between four trees related
to the complex $\Comp$. The first is the tree $\T_\K = \Comp_2$ with the action of
$\Aut(\K^2)$. We will define a second tree $\T(v_2) \subset \Comp$ associated with each type 2
vertex $v_2$, isomorphic to~$\T_\K$. Thirdly, we have the tree $\T_{\K(x)}$ (which is the complex $\Comp_2$ constructed starting from the field $\K(x)$ instead of $\K$). Finally, we will have a tree $\T(v_1)$ associated to each
type 1 vertex $v_1$, appearing in a sequence of simplicial projections $\LL(v_1) \to \T(v_1) \to \T_{\K(x)}$, where $\LL(v_1)$
is the link of $v_1$.

\subsection{Tree associated with a type 2 vertex}

For any vertex $v_2$ of $\Comp$ of type $2$ we define the following tree
$\T(v_2) \subset \Comp^1$.
We require the definition to be $\TA$-equivariant, and hence by the transitivity of the action we can assume $v_2 = [x_1,x_2]$.
Then we consider all vertices of the form $[f_1(x_1,x_2)]$ and $[f_1(x_1,x_2), f_2(x_1,x_2)]$, where
$(f_1, f_2) \in \Tame(\A^2)$.
These vertices span a tree $\T({v_2})$ isomorphic to $\T_\K$.
Observe that the embedding $\T(v_2) \subset \Comp$ is not an isometric embedding, since $\T(v_2)$ is contained in the link of a type 1 vertex $v_1$ (for example, if $v_2 = [x_1, x_2]$, then one can take $v_1 = [x_3]$).

We now prove a preparatory lemma about leading terms in the vertices of the
tree $\T_\K$.

\begin{lemma}
\label{lem:leading terms}
Let $[f]$ be a type 1 vertex in $\T_\K$, and consider the connected components of the tree $\T_\K$ with the vertex $[x_1, x_2]$ removed.
\begin{enumerate}[wide]
\item For the unique $(a:b) \in \mathbb P^1_\K$ such that $[f]$ and $[ax_1 + bx_2]$ belong to the same component, we have
$$f(x_1,x_2) = c(ax_1 + bx_2)^d + R(x_1,x_2),$$
with $c \in \K^*$, $d \ge 1$, and
$\deg R<d$.

\item Assume that $[f]\neq [x_2]$ belong to the same component.
Then
$$f(x_1,x_2) = Q(x_2)x_2^{d+1} + cx_1 x_2^d + R(x_1,x_2),$$
with $c \in \K^*$, $d \ge 0$, $\deg R<d+1$, and $Q$ nonconstant.

In particular, if $A \in \K[X]$ is any nonconstant polynomial, then $A(f(x_1,
x_2))$ has the same form.
\end{enumerate}
\end{lemma}

\begin{proof}
By definition, vertices in the link of $[x_1, x_2]$ are of the form $[ax_1 +
bx_2]$.
Moreover, since the stabiliser of $[x_1, x_2]$ is isomorphic to $A_2 = \GL_2(\K)
\rtimes \K^2$, it acts transitively on these vertices.
So it is sufficient to consider the case where $[f]$ belongs to the same
component as $[x_2]$, and to prove assertion (2).
In fact we are going to prove that the consecutive type~1 vertices $[f_0] = [x_2], \ldots,
[f_n] = [f]$ in the unique embedded path from $[x_2]$ to $[f]$ are of the
form (for $i = 1, \dots, n$)
\begin{equation}f_i(x_1,x_2) = Q_i(x_2)x_2^{d_i+1} + c_ix_1 x_2^{d_i} + \text{lower order terms},\tag{$\star$}
\label{eq:leading terms}
\end{equation}
with the exponents $d_i$ strictly increasing with $i$, and $Q_i$ nonconstant.

Observe that from Section~\ref{sec:complex} it follows that for vertices $[g_1],[g_3]$ in $\T_\K$ joined by an embedded path of length four with centre $[g_2]$, the vertex $[g_3]$ has the form $[g_1 + P(g_2)]$, for some
polynomial $P(X)$ that is non-affine, i.e.\ not of the form $aX + b$.

Applying this observation to $g_1=x_1, g_2=x_2,g_3=f_1$, we obtain (\ref{eq:leading terms})  for $i=1$, with $d_1 = 0$.
Then we proceed by induction and assume that for some $j\geq 2$, we proved (\ref{eq:leading terms}) for all $i<j$.
Then applying the observation above to $g_1=f_{j-2}, g_2=f_{j-1},g_3=f_j$, we arrive at (\ref{eq:leading terms}) for $i=j$, with $d_j > d_{j-1}$.
\end{proof}

\subsection{Links}

It is an easy observation (see \cite[\S1.C]{Lamy}) that the link $\LL(v_2)$ of a vertex of type $2$ in the
complex $\Comp$ is a full bipartite graph, and the link $\LL(v_3)$ of a vertex of type $3$ is the incidence graph of the projective plane $\mathbb P^2_\K$.
Now we study the link $\LL(v_1)$ of a vertex of type $1$ in $\Comp$.
We know from \cite[Lem 5.6]{Lamy} that $\LL(v_1)$ is connected.
By transitivity of the action we can assume $v_1 = [x_3]$.
We describe two trees related to $\LL([x_3])$.

Let $\pi$ be the simplicial map from $\LL([x_3])$ to the tree $\T_{\K(x_3)}$
(see Section~\ref{sec:complex})
mapping $[f_1, x_3] $ to $[f_1]$ and $[f_1, f_2, x_3]$ to $[f_1,f_2]$. The map
$\pi$ is $\Stab([x_3])$-equivariant.
Since $\LL([x_3])$ is connected, the image $\pi
(\LL([x_3]))$ is a subtree of $\T_{\K(x_3)}$. The group $\Stab([x_3])$ acts
transitively on the edges of $\LL([x_3])$,
hence transitively on the edges of $\pi (\LL([x_3]))$.
Denote by $H_1\subset \Stab([x_3])$ the stabiliser of $\pi^{-1}([x_1])$ and by $H_2\subset\Stab([x_3])$ the stabiliser
of $\pi^{-1}([x_1, x_2])$.
By Bass--Serre theory we obtain:

\begin{lemma}
\label{lem:lem1}
$$\Stab([x_3]) = H_1 *_{H_1 \cap H_2} H_2.$$
\end{lemma}

We can give explicit descriptions of $H_1$ and $H_2$ (compare with
\cite[\S4.1]{BFL}).

\begin{lemma}
\label{lem:lem2}
\begin{align*}
H_1 &= \left\lbrace \left(ax_1 + t(x_3),\, dx_2 + P(x_1, x_3),\,
\alpha x_3 + \beta \right)\ | \ a,d \in \K^* \right\rbrace,\\
H_2 &= \left\lbrace \left( a(x_3)x_1 + b(x_3)x_2 + t(x_3),\, c(x_3)x_1 +
d(x_3)x_2 + t'(x_3),\,
\alpha x_3 + \beta \right) \right\rbrace \\
&\simeq \left( \GL_2(\K[x_3]) \ltimes \K[x_3]^2 \right) \rtimes A_1,\\
H_1 \cap H_2 &= \left\lbrace \left(ax_1 + t(x_3),\, dx_2 + c(x_3)x_1 +
t'(x_3),\,
\alpha x_3 + \beta \right)\ | \ a,d \in \K^* \right\rbrace.
\end{align*}
\end{lemma}

\begin{proof}
Take $[f_1, x_3] \in  \pi^{-1}([x_1])$, that is, $[f_1] = [x_1]$ in the tree
$\T_{\K(x_3)}$.
So $f_1 = a(x_3) x_1 + t(x_3)$ for some polynomials $a,t$, and moreover by
definition there exists $f_2 \in \K[x_1, x_2, x_3]$ such that $f = (f_1, f_2,
x_3) \in \TA$.
In particular the Jacobian of $f$ is a multiple of $a$, so $a$ must be a
constant.
We obtain
$$  \pi^{-1}([x_1])= \left\lbrace [ax_1 + t(x_3), x_3]\ | \  a \in \K^*, t \in
\K[x_3] \right\rbrace.$$
Let $h = (h_1, h_2, h_3) \in H_1 = \Stab( \pi^{-1}([x_1]) )$. Since $h\in \Stab([x_3])$, we have $h_3 = \alpha x_3 + \beta$.
Moreover, $h^{-1}\cdot[x_1,x_3] = [h_1, x_3]$, so $h_1 = ax_1 + t(x_3)$ for some $a
\in \K^*$ and $t \in \K[x_3]$.
Then we write $h_2 = \sum_{i \ge 0} c_i(x_1, x_3) x_2^i$, and compute the
Jacobian of $h$, which must be a constant:
$$\Jac(h) = \left\lvert
\begin{matrix}
a & 0 & t'(x_3) \\
* & \sum_{i \ge 1} ic_i x_2^{i-1} & * \\
0 & 0 & \alpha
\end{matrix}
\right\rvert
= a \alpha \sum_{i \ge 1} ic_i x_2^{i-1}$$
We obtain $c_1 \in \K^*$ and $c_i = 0$ for all $i \ge 2$, hence the expected
expression $h_2 = dx_2 + P(x_1, x_3)$, where $P = c_0$ and $d = c_1$.

The proof for the form of elements $h =(h_1, h_2, h_3) \in H_2$ is in fact
easier. Indeed, we have $[h_1, h_2] = [x_1, x_2]$ in $\T_{\K(x_3)}$ if and only if $(h_1, h_2)$
is an element in the affine group over $\K(x_3)$, as desired.
\end{proof}

Using Lemma~\ref{lem:lem2} and the classical theorem of Nagao
\cite[\S1.6]{Serre} about the amalgamated product structure of
$\GL_2(\K[x_3])$, we obtain (proof left to the reader, see
\cite[Lem 4.8]{BFL} for a very similar result):

\begin{lemma}
\label{lem:lem3}
$$H_2 = K_1 *_{K_1 \cap K_2} K_2$$
where $K_2 = \left( \GL_2(\K) \ltimes \K[x_3]^2 \right) \rtimes A_1$ and $K_1 =
\left( \left\lbrace \begin{pmatrix} a & 0 \\ c(x_3) & d \end{pmatrix}
\right\rbrace  \ltimes \K[x_3]^2 \right) \rtimes A_1$.
\end{lemma}

Finally, observing that $H_1 \cap H_2 = K_1$, and consequently $H_1\cap
K_2=K_1\cap K_2$, by combining Lemmas~\ref{lem:lem1} and~\ref{lem:lem3} we
obtain an alternative decomposition of $\Stab([x_3])$:

\begin{prop}
\label{prop:maintree}
$$\Stab([x_3]) = H_1 *_{H_1 \cap K_2} K_2.$$
\end{prop}
\noindent Denote by  $\T([x_3])$ the Bass--Serre tree of
Proposition~\ref{prop:maintree}.
Since $K_2$ is a supergroup in $\Stab([x_3])$ of the stabiliser of
$[x_1,x_2,x_3]$, the tree $\T([x_3])$ admits a projection
$\sigma :  \LL([x_3]) \to \T([x_3])$
such that the following diagram commutes:
$$\xymatrix{\LL([x_3]) \ar[rr]^\pi \ar[dr]_\sigma & & \T_{\K(x_3)} \\
 & \T([x_3]) \ar[ur]}$$

\begin{rem}
\label{rem:same sigma}
A type 3 vertex $v_3 \in \LL([x_3])$ satisfies $\sigma(v_3) = \sigma([x_1, x_2,
x_3])$ if and only if $v_3 = [x_1 + t(x_3), x_2 + t'(x_3), x_3]$ for some $t,
t' \in \K[x_3]$.
Indeed, by definition $v_3$ is the image of $[x_1, x_2,
x_3]$ under the action of an element of $K_2$.
Since a representative of $v_3$ is defined up to an affine map, and  $
\GL_2(\K) \times A_1$ is a subgroup of the affine group $A_3$, we get the
above form for $v_3$.
\end{rem}

\section{Diagrams and Curvature}
\label{sec:curvature}

A \emph{disc diagram} $\D$ is a combinatorial (not necessarily simplicial)
complex built of triangles, which is homeomorphic to a disc.
A \emph{disc diagram in $\Comp$} is a combinatorial map of a disc diagram $\D$
into $\Comp$ (not necessarily an embedding).
In that case a vertex of $\D$ has \emph{type $i$}, if its image in $\Comp$ has
type $i$. A disc diagram in $\Comp$ is \emph{reduced}, if it is a local
embedding at the interior points of edges and at interior vertices of type~$1$. We
similarly define \emph{reduced sphere diagrams}.

\begin{lemma}
\label{lem:disc_exist}
Each embedded loop in $\Comp^{1}$ bounds a
reduced disc diagram.
\end{lemma}
\begin{proof}
By \cite[Prop 5.7]{Lamy}, the complex $\Comp$ is simply
connected. Thus by the relative simplicial approximation theorem, there is a
disc diagram $\D\to \Comp$ with prescribed $\partial \D$.
Moreover, if $\D$ has minimal area, then it is automatically simplicial and
$\D\to \Comp$ is a local embedding at edges. Finally,
to obtain that $\D\to \Comp$ is a local embedding at interior type 1 vertices, it
suffices to perform the following operation on $\D$.
For each type 1 vertex $v$ in $\D$ with edges $e_1,\ldots,e_n$ starting at $v$
that are mapped to the same edge of type $(1,i)$ in $\Comp$, we cut $\D$ along $e_1\cup \ldots
\cup e_n$ and reglue the arising $2n$-gon so that all type $i$ vertices become
identified. Repeating this operation finitely many times and keeping the notation $\D$ for the reglued disc, we obtain that
$\D\to \Comp$ is reduced. Note that $\D$ is still a topological disc, since we assumed that $\partial \D$ embeds.
\end{proof}

\subsection{Arrows}

Let $\D\rightarrow \Comp$ be a reduced disc diagram in $\Comp$.
Let $T, T'$ be triangles in~$\D$ with a common edge $e$ of type $(1,2)$. We will
now describe, when do we equip this edge with an arrow (an orientation). We
require the definition to be $\TA$-equivariant and translate $T,T'$ to the
following pair.

\begin{lemma}[see {\cite[Cor 1.5]{Lamy}}] \label{lem:adjacent}
Two triangles in $\Comp$ adjacent along an edge of type $(1,2)$ can be sent by an element of $\TA$ to
\begin{align*}
T &= [x_1, x_2, x_3], [x_2, x_3], [x_3] \\
T' &= [x_1+ P(x_2, x_3), x_2, x_3], [x_2, x_3], [x_3]
\end{align*}
with $P \in \K[x_2, x_3]$ of degree at least 2.
\end{lemma}

In the situation of Lemma~\ref{lem:adjacent}, if there exists a vertex
$v_1=[f(x_2, x_3)] \in \T([x_2, x_3])$ and a polynomial $A(X) \in \K[X]$ such that $P(x_2,x_3) =
A(f(x_2, x_3))$, we say that the common edge $e$ between $T$ and $T'$ is
\emph{oriented}.
By Lemma \ref{lem:leading terms}(1), we know that $P$ has the following form:
$$P(x_2,x_3) = c(ax_2 + bx_3)^d + \text{ lower other terms}$$

To specify the orientation, we put an \emph{arrow} on the edge $e$ pointing
towards the connected component of $\T_{[x_2,x_3]} - e$ containing
$[ax_2 + bx_3]$ (see Figure~\ref{fig:arrows}).
Moreover, if $P$ is a polynomial in $x_3$ (up to a linear term in $x_2$), then
we say that the arrow on~$e$ is \emph{terminal} and use a double arrow in
figures. However, in future figures a single arrow might also represent a
terminal one.
If $e$ is nonoriented, we use a wavy edge in figures.

\begin{rem}
In the definition of orientation, we do not claim that there is a unique way to
write $P(x_2,x_3) = A(f(x_2, x_3))$.
Indeed, consider the following example, where $d \ge 2$:
\begin{align*}
T' &= [x_1+ x_2 + x_3^d, x_2, x_3], [x_2, x_3], [x_3] \\
  &= [x_1+ x_3^d, x_2, x_3], [x_2, x_3], [x_3]
\end{align*}
One can treat $T'$ as $[x_1 + A(f), x_2,x_3],
[x_2, x_3], [x_3]$ in two ways:
\begin{align*}
\text{either } A(X) &= X, & f &= x_2 + x_3^d; \\
\text{or } A(X) &= X^d, & f &= x_3.
\end{align*}
However, the leading term of $A(f)$ is uniquely defined, independently of such a choice.
\end{rem}

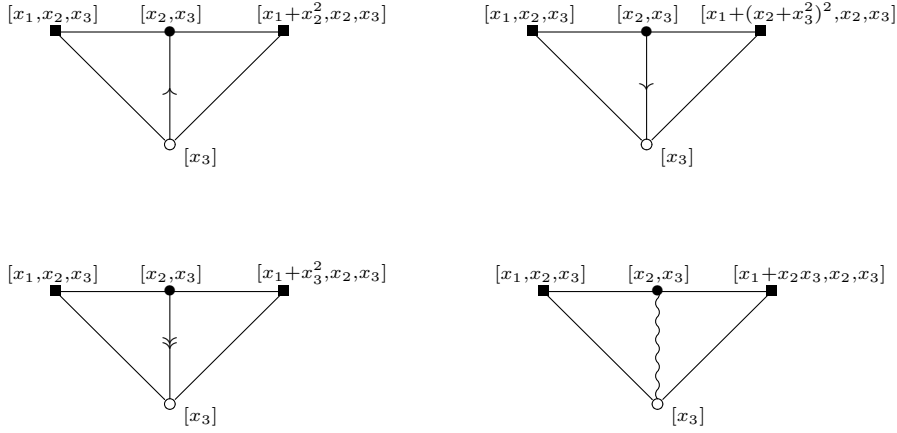
\begin{figure}
$$
\xymatrix{
\mygraph{
!{<0cm,0cm>;<3cm,0cm>:<0cm,1.5cm>::}
!{(1.5,1)}*-{\typethree}="45"
!{(1,1)}*-{\typetwo}="5"
!{(.5,1)}*-{\typethree}="15"
!{(1,0)}*{\typeone}="0"
"45"-_<(-.3){[x_1+x_2^2,x_2,x_3]}"5"-_<{[x_2,x_3]}_>{[x_1,x_2,x_3]}"15"
"0"-|@{>}"5"
"0"-_<{[x_3]}"45" "0"-"15"
}
&
\mygraph{
!{<0cm,0cm>;<3cm,0cm>:<0cm,1.5cm>::}
!{(1.5,1)}*-{\typethree}="45"
!{(1,1)}*-{\typetwo}="5"
!{(.5,1)}*-{\typethree}="15"
!{(1,0)}*{\typeone}="0"
"45"-_<(-.3){[x_1+(x_2+x_3^2)^2,x_2,x_3]}"5"-_<{[x_2,x_3]}_>{[x_1,x_2,x_3]}"15"
"0"-|@{<}"5"
"0"-_<{[x_3]}"45" "0"-"15"
}
\\
\mygraph{
!{<0cm,0cm>;<3cm,0cm>:<0cm,1.5cm>::}
!{(1.5,1)}*-{\typethree}="45"
!{(1,1)}*-{\typetwo}="5"
!{(.5,1)}*-{\typethree}="15"
!{(1,0)}*{\typeone}="0"
"45"-_<(-.3){[x_1+x_3^2,x_2,x_3]}"5"-_<{[x_2,x_3]}_>{[x_1,x_2,x_3]}"15"
"0"-|@{<<}"5"
"0"-_<{[x_3]}"45" "0"-"15"
}
&
\mygraph{
!{<0cm,0cm>;<3cm,0cm>:<0cm,1.5cm>::}
!{(1.5,1)}*-{\typethree}="45"
!{(1,1)}*-{\typetwo}="5"
!{(.5,1)}*-{\typethree}="15"
!{(1,0)}*{\typeone}="0"
"45"-_<(-.3){[x_1+x_2x_3,x_2,x_3]}"5"-_<{[x_2,x_3]}_>{[x_1,x_2,x_3]}"15"
"0"-@{~}"5"
"0"-_<{[x_3]}"45" "0"-"15"
}
}
$$
\caption{Examples of arrows} \label{fig:arrows}
\end{figure}

\begin{lemma}
\label{lem:inarrow}
Let $T = v_1, v_2, v_3$ and $T' = v_1, v_2, v_3'$ be two triangles in $\D$
adjacent
along an edge of type $(1,2)$. Consider the map $\sigma\colon
\LL(v_1) \to \T(v_1)$ defined in Section~\ref{sec:trees}.
Then there is a terminal arrow from $v_2$ to $v_1$ if and only if $\sigma(v_3)
= \sigma(v_3')$.
\end{lemma}

\begin{proof}
By Lemma~\ref{lem:adjacent}, we can assume $v_1 = [x_3], v_2 = [x_2, x_3], v_3 =
[x_1, x_2, x_3]$, and $v_3' = [x_1 + P(x_2, x_3), x_2, x_3]$.
By Remark \ref{rem:same sigma},
$\sigma(v_3) = \sigma(v_3')$ means that $P \in \K[x_3]$ (up to a linear term in
$x_2$), which is precisely the definition of a terminal arrow.
\end{proof}

\subsection{Curvature}

In this subsection we define the \emph{curvature} at each vertex of a reduced
disc diagram $\D$ in $\Comp$.

First we consider an interior vertex $v_i$ of $\D$.
By $\deg v_i$ we denote its valence in~$\D$, and by $\out v_i $ and $\inn v_i $
the number of outgoing and incoming arrows (a terminal arrow is only counted
once).
The definition of the curvature $K(v_i)$ depends on the type~$i$ of the vertex:

\begin{align*}
K(v_1) &= \frac{1}{12}(12 - \deg v_1) + \frac{\out v_1 - \inn v_1 }{6} ,\\
K(v_2) &= \frac{1}{4}(4 - \deg v_2) + \frac{\out v_2  - \inn v_2 }{6}, \\
K(v_3) &= \frac{1}{6}(6 - \deg v_3).
\end{align*}

One should imagine the triangles of $\Comp$ being equipped with angles $\frac{\pi}{6},\frac{\pi}{2},\frac{\pi}{3}$ with a correction coming from the arrows.

\begin{rem}
\label{rem:negativeminus1over6} The valence of $v_i$ is always even.
Hence if $K(v_i)$ is negative, then in fact $K(v_i) \leq -\frac{1}{6}$.
\end{rem}

Similarly we define the boundary curvature $K_\partial(v_i)$ of a boundary vertex $v_i$:

\begin{align*}
K_\partial(v_1) &= \frac{1}{12}(7 - \deg v_1) + \frac{\out v_1  - \inn v_1 }{6}, \\
K_\partial(v_2) &= \frac{1}{4}(3 - \deg v_2) + \frac{\out v_2  - \inn v_2 }{6}, \\
K_\partial(v_3) &= \frac{1}{6}(4 - \deg v_3).
\end{align*}

\begin{rem}
\label{rem:5over6}
For $i=2,3$, the boundary curvature $K_\partial(v_i)$ is maximal when $v_i$ has valence 2, giving $K_\partial(v_i)\leq \frac{1}{3}$. For $i=1$, note that only every other additional edge might have an outgoing arrow. Hence the contribution $\frac{1}{6}$ from the $\out$ term of such an edge is cancelled with the contribution $-2\frac{1}{12}$ from the $\deg$ term. Thus $K_\partial(v_1)$ is maximal when $\deg v_1=2k+1$ and $\out v_1 =k$ giving $K_\partial(v_1) \leq \frac{1}{2}$.
\end{rem}

As in classical small cancellation (see \cite[Thm V.3.1]{LS}), we obtain the following.

\begin{prop}(Combinatorial Gauss--Bonnet)
\label{prop:GB}
For any reduced disc diagram~$\D$ in~$\Comp$, the sum of the curvatures and boundary
curvatures of its vertices equals~$1$. Similarly, for any reduced sphere diagram in
$\Comp$, the sum of the curvatures of its vertices equals~$2$.
\end{prop}

\begin{proof}
First of all, since the out and the in terms cancel out in doing the sum, it suffices to prove the theorem when in the definition of the curvature these terms are omitted.
Denote by $w_i,e_{ij},f$ the numbers of vertices of type $i$, edges of type
$(i,j)$, and faces in $\D$. The Euler characteristic of $\D$ can be expressed as
\begin{multline*}
1 = v_1+v_2+v_3-e_{12}-e_{13}-e_{23}+f\\
= \Big(v_1-\frac{6(e_{12}+e_{13})}{12}+\frac{5f}{12}\Big)+\Big(v_2-\frac{2(e_{12}+e_{23})}{4}+\frac{f}{4}\Big)+
\Big(v_3-\frac{3(e_{13}+e_{23})}{6}+\frac{2f}{6}\Big).
\end{multline*}
We claim that each of these terms can be recognised as the sum of the curvatures
and boundary curvatures at the vertices of given type. Indeed, let $V_1$ be the
set of all vertices of type $1$ in $\D^0$. For each $v\in V_1$, let $e^v,f^v$ be
the number of edges and faces incident to $v$. If $v$ is interior, then
$K(v)=1-\frac{\deg v}{12}$ after ignoring the out and the in terms. Thus
$$K(v)=1-\frac{6\deg v}{12}+\frac{5\deg
v}{12}=1-\frac{6e^v}{12}+\frac{5f^v}{12}.$$
If $v\in \partial \D$, we have $K_\partial(v)=1-\frac{\deg v}{12}-\frac{5}{12}$,
and thus we obtain the same formula
$$K_\partial(v)=1-\frac{6\deg v}{12}+\frac{5(\deg v-1)}{12}=1-\frac{6e^v}{12}+\frac{5f^v}{12}.$$
Moreover, we have $\sum_{v\in V_1}e^v=e_{12}+e_{13}$ and $\sum_{v\in V_1}f^v=f$. This proves the claim for the vertices of type $1$. The other types are dealt with similarly. The claim immediately implies the proposition. The sphere case follows in the same way.
\end{proof}

\section{Nonpositive curvature}
\label{sec:nonpositive}

In this section we study the curvature at interior vertices of reduced disc diagrams in $\Comp$, and in particular we prove  Theorem~\ref{thm:contractible}.

\begin{prop}
\label{prop:nonpositivecurvature}
The curvature at an interior vertex of a reduced disc diagram is nonpositive, and equal to zero only in one of the eight
situations in Figure~\ref{fig:8cases}.
\end{prop}

An arrow shown standard in Figure~\ref{fig:8cases} is allowed to be terminal,
and the edges on the boundary of the diagrams in Figure~\ref{fig:8cases}(a)
might or might not carry arrows.

\begin{rem}
\label{rem:outgoing}
In each of the diagrams in Figure~\ref{fig:8cases}(d,e,f,g,h), given any 3
consecutive edges of type $(1,2)$ from the central vertex, at least one carries
an incoming arrow.

\end{rem}

\begin{rem}\label{rem:outgoingbis}
Among the diagrams in Figure~\ref{fig:8cases}(d,e,f,g,h), only in cases (g,h) there are two edges from the central vertex carrying
outgoing arrows, and their distance in its link is 2 or 6 (and not 4).
\end{rem}

\begin{figure}[t]
$$
\xymatrix@R-1pc@C-1pc{
(a) &
\mygraph{
!{<0cm,0cm>;<1cm,0cm>:<0cm,.8cm>::}
!{(-1,1)}*{\typeone}="NW"
!{(0,1)}*-{\typetwo}="N"
!{(1,1)}*{\typeone}="NE"
!{(-.5,0)}*-{\typetwo}="W"
!{(0,0.33)}*-{\typethree}="O"
!{(.5,0)}*-{\typetwo}="E"
!{(0,-1)}*{\typeone}="S"
"NW"-"N"-"NE"-"E"-"S"-"W"-"NW"
"NW"-"O" "NE"-"O" "S"-"O"
"N"-"O" "W"-"O"-"E"
}
& \qquad (e)  &
\mygraph{
!{<0cm,0cm>;<.8cm,0cm>:<0cm,.8cm>::}
!{(-1,1)}*-{\typetwo}="NW"
!{(0,1)}*-{\typethree}="N"
!{(1,1)}*-{\typetwo}="NE"
!{(-1,0)}*-{\typethree}="W"
!{(0,0)}*{\typeone}="O"
!{(1,0)}*-{\typethree}="E"
!{(-1,-1)}*-{\typetwo}="SW"
!{(0,-1)}*-{\typethree}="S"
!{(1,-1)}*-{\typetwo}="SE"
"NW"-"N"-"NE"-"E"-"SE"-"S"-"SW"-"W"-"NW"
"NW"-|@{>>}"O" "SE"-|@{>>}"O"
"NE"-@{~}"O" "SW"-@{~}"O"
"N"-"O"-"S" "W"-"O"-"E"
}
\\
(b) &
\mygraph{
!{<0cm,0cm>;<1cm,0cm>:<0cm,.8cm>::}
!{(0,1)}*{\typeone}="N"
!{(-1,0)}*-{\typethree}="W"
!{(0,0)}*-{\typetwo}="O"
!{(1,0)}*-{\typethree}="E"
!{(0,-1)}*{\typeone}="S"
"N"-"E"-"S"-"W"-"N"
"N"-@{~}"O"-@{~}"S" "W"-"O"-"E"
}
& \qquad (f) &
\mygraph{
!{<0cm,0cm>;<.8cm,0cm>:<0cm,.8cm>::}
!{(-1,1)}*-{\typetwo}="NW"
!{(0,1)}*-{\typethree}="N"
!{(1,1)}*-{\typetwo}="NE"
!{(-1,0)}*-{\typethree}="W"
!{(0,0)}*{\typeone}="O"
!{(1,0)}*-{\typethree}="E"
!{(-1,-1)}*-{\typetwo}="SW"
!{(0,-1)}*-{\typethree}="S"
!{(1,-1)}*-{\typetwo}="SE"
"NW"-"N"-"NE"-"E"-"SE"-"S"-"SW"-"W"-"NW"
"NW"-|@{>>}"O" "SE"-|@{>>}"O"
"NE"-|@{<}"O"  "SW"-|@{>}"O"
"N"-"O"-"S" "W"-"O"-"E"
}\\
(c) &
\mygraph{
!{<0cm,0cm>;<1cm,0cm>:<0cm,.8cm>::}
!{(0,1)}*{\typeone}="N"
!{(-1,0)}*-{\typethree}="W"
!{(0,0)}*-{\typetwo}="O"
!{(1,0)}*-{\typethree}="E"
!{(0,-1)}*{\typeone}="S"
"N"-"E"-"S"-"W"-"N"
"N"-|@{<}"O"-|@{<}"S" "W"-"O"-"E"
}& \qquad (g) &
\mygraph{
!{<0cm,0cm>;<.8cm,0cm>:<0cm,.8cm>::}
!{(-1.1,.5)}*-{\typethree}="NWW"
!{(-.7,1)}*-{\typetwo}="NW"
!{(0,1)}*-{\typethree}="N"
!{(.7,1)}*-{\typetwo}="NE"
!{(1.1,.5)}*-{\typethree}="NEE"
!{(-1.5,0)}*-{\typetwo}="W"
!{(0,0)}*{\typeone}="O"
!{(1.5,0)}*-{\typetwo}="E"
!{(-1.1,-.5)}*-{\typethree}="SWW"
!{(-.7,-1)}*-{\typetwo}="SW"
!{(0,-1)}*-{\typethree}="S"
!{(.7,-1)}*-{\typetwo}="SE"
!{(1.1,-.5)}*-{\typethree}="SEE"
"NW"-"N"-"NE"-"E"-"SE"-"S"-"SW"-"W"-"NW"
"NW"-|@{<}"O" "SE"-@{~}"O"
"NE"-|@{<}"O"  "SW"-@{~}"O"
"N"-"O"-"S" "W"-|@{>>}"O"-|@{<<}"E"
"NWW"-"O"-"SWW" "NEE"-"O"-"SEE"
}\\
(d) &
\mygraph{
!{<0cm,0cm>;<1cm,0cm>:<0cm,.8cm>::}
!{(-1,1)}*-{\typetwo}="NW"
!{(0,1)}*-{\typethree}="N"
!{(1,1)}*-{\typetwo}="NE"
!{(-.5,0)}*-{\typethree}="W"
!{(0,0.33)}*{\typeone}="O"
!{(.5,0)}*-{\typethree}="E"
!{(0,-1)}*{\typetwo}="S"
"NW"-"N"-"NE"-"E"-"S"-"W"-"NW"
"NW"-|@{>>}"O" "NE"-|@{>>}"O" "S"-|@{>>}"O"
"N"-"O" "W"-"O"-"E"
}& \qquad (h) &
\mygraph{
!{<0cm,0cm>;<.8cm,0cm>:<0cm,.8cm>::}
!{(-1.1,.5)}*-{\typethree}="NWW"
!{(-.7,1)}*-{\typetwo}="NW"
!{(0,1)}*-{\typethree}="N"
!{(.7,1)}*-{\typetwo}="NE"
!{(1.1,.5)}*-{\typethree}="NEE"
!{(-1.5,0)}*-{\typetwo}="W"
!{(0,0)}*{\typeone}="O"
!{(1.5,0)}*-{\typetwo}="E"
!{(-1.1,-.5)}*-{\typethree}="SWW"
!{(-.7,-1)}*-{\typetwo}="SW"
!{(0,-1)}*-{\typethree}="S"
!{(.7,-1)}*-{\typetwo}="SE"
!{(1.1,-.5)}*-{\typethree}="SEE"
"NW"-"N"-"NE"-"E"-"SE"-"S"-"SW"-"W"-"NW"
"NW"-|@{<}"O" "SE"-|@{<}"O"
"NE"-@{~}"O"  "SW"-@{~}"O"
"N"-"O"-"S" "W"-|@{>>}"O"-|@{<<}"E"
"NWW"-"O"-"SWW" "NEE"-"O"-"SEE"
}
}
$$
\caption{The eight cases of Proposition \ref{prop:nonpositivecurvature}} \label{fig:8cases}
\end{figure}
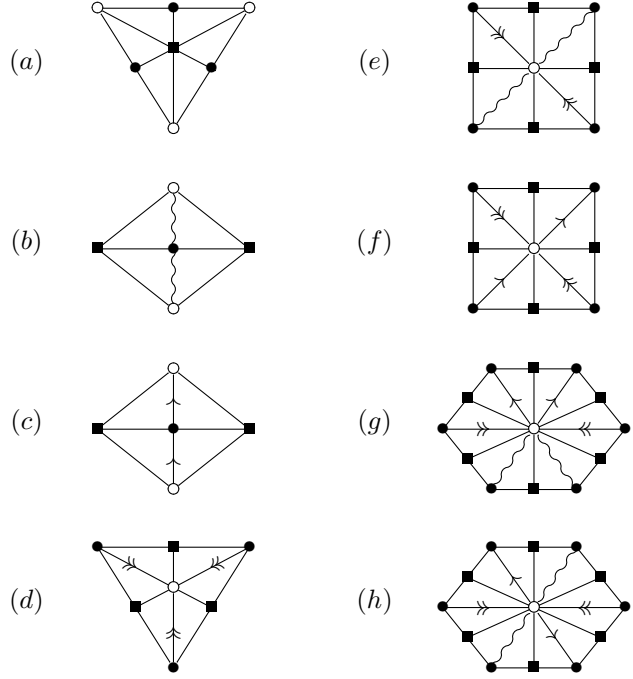

To prove Proposition~\ref{prop:nonpositivecurvature} we need the following preparatory lemmas.

\begin{lemma}
\label{lem:type2weird}
There is no reduced disc diagram as in Figure~\ref{fig:type2}(a).
\end{lemma}

\begin{proof}
By contradiction, assume that such a diagram exists.
Using the action of $\TA$ we can assume that the central vertex of the diagram is $[x_2, x_3]$, the three type 1 vertices are $[x_2]$, $[x_3]$ and $[x_2 + x_3]$, and one of the type 3 vertices is $[x_1, x_2, x_3]$.
Then the arrows imply that the remaining type 3 vertices have the form $[x_1+A(f), x_2, x_3]$ and $[x_1-B(g), x_2, x_3]$, with a relation of the form
\begin{equation}\label{eq:ABC}
A(f) + B(g) = C(h),
\end{equation}
for some components $f, g,h \in \K[x_2, x_3]$ and some polynomials $A, B,C \in \K[X]$.
%
By Lemma \ref{lem:leading terms}(1), we know that the leading monomials of $A(f)$, $B(g)$ and $C(h)$ have the form $x_2^a$, $x_3^b$, $(x_2 + x_3)^c$.
Moreover, since the type 3 vertices are distinct, we have $a,b,c \ge 2$.
This is incompatible with relation (\ref{eq:ABC}).
\end{proof}

\begin{lemma}
\label{lem:type2strange}
There are no reduced disc diagrams as in Figure~\ref{fig:type2}(b,c).
\end{lemma}
\begin{proof}
We can suppose that  the central vertex $v_2$ of the diagram is $[x_2,x_3]$ and that its type $3$ neighbours are $[x_1,x_2,x_3]$ and $[x_1+P(x_2,x_3),x_2,x_3]$.
Let $e$ be an edge from~$v_2$ with an outgoing arrow, and let $e'$ be the other type $(1,2)$ edge (possibly also with an arrow).
Hence there is a type 1 vertex $v_1=[f(x_2,x_3)]$ with $P=A(f)$. Both $e$ and $e'$ lie in $\T(v_2)$, and carry arrows pointing towards the same vertex $v_1\in \T(v_2)$, which depends only on $[x_1,x_2,x_3]$ and $[x_1+P(x_2,x_3),x_2,x_3]$.
This excludes the diagrams in Figure~\ref{fig:type2}(b,c).
\end{proof}

\begin{figure}[t]
$$
\xymatrix@R-1pc{
\mygraph{
!{<0cm,0cm>;<1cm,0cm>:<0cm,.8cm>::}
!{(-1,1)}*{\typeone}="NW"
!{(0,1)}*-{\typethree}="N"
!{(1,1)}*{\typeone}="NE"
!{(-.5,0)}*-{\typethree}="W"
!{(0,0.33)}*-{\typetwo}="O"
!{(.5,0)}*-{\typethree}="E"
!{(0,-1)}*{\typeone}="S"
"NW"-"N"-"NE"-"E"-"S"-"W"-"NW"
"NW"-|@{<}"O" "NE"-|@{<}"O" "S"-|@{<}"O"
"N"-"O" "W"-"O"-"E"
}
&
\mygraph{
!{<0cm,0cm>;<1cm,0cm>:<0cm,.8cm>::}
!{(0,1)}*{\typeone}="N"
!{(-1,0)}*-{\typethree}="W"
!{(0,0)}*-{\typetwo}="O"
!{(1,0)}*-{\typethree}="E"
!{(0,-1)}*{\typeone}="S"
"N"-"E"-"S"-"W"-"N"
"N"-|@{<}"O"-|@{>}"S" "W"-"O"-"E"
}
&
\mygraph{
!{<0cm,0cm>;<1cm,0cm>:<0cm,.8cm>::}
!{(0,1)}*{\typeone}="N"
!{(-1,0)}*-{\typethree}="W"
!{(0,0)}*-{\typetwo}="O"
!{(1,0)}*-{\typethree}="E"
!{(0,-1)}*{\typeone}="S"
"N"-"E"-"S"-"W"-"N"
"N"-|@{<}"O"-@{~}"S" "W"-"O"-"E"
}\\
(a) & (b) & (c)
}
$$
\caption{Three impossible cases around a type 2 vertex} \label{fig:type2}
\end{figure}
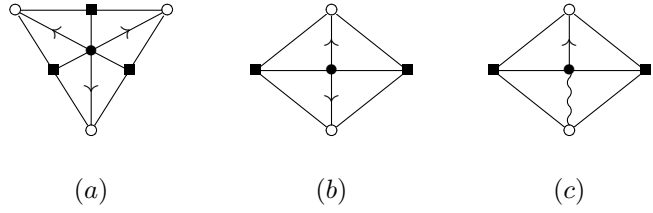

In the following lemma we discuss configurations in the link of a type 1 vertex~$v_1$ with a specified image in $\T(v_1)$ under $\sigma$ from Section~\ref{sec:trees}.

\begin{lemma} \phantomsection\label{lem:impossible}
\begin{enumerate}[(a), wide]
\item In a reduced disc diagram as in Figure~\ref{fig:impossible}(a), if the left
horizontal edge has an arrow outgoing from the central vertex,
then the other horizontal edge has an incoming arrow.
\item
There is no reduced disc diagram with a type 1 vertex $v_1$ that has a pair of
disjoint length 2 paths in the link mapping to the same length 2 path in
$\T(v_1)$, and equipped with a pair of arrows as in
Figure~\ref{fig:impossible}(b).
\end{enumerate}
\end{lemma}

\begin{figure}
$$
\xymatrix@R-1pc{
\mygraph{
!{<0cm,0cm>;<1cm,0cm>:<0cm,1cm>::}
!{(0,0)}*{\typeone}="0"
!{(0,-1.5)}*-{\typetwo}="1"
!{(0,1.5)}*-{\typetwo}="3"
!{(-1,-1)}*{\typethree}="12"
!{(-1,0)}*-{\typetwo}="2"
!{(-1,1)}*{\typethree}="23"
!{(1,1)}*{\typethree}="34"
!{(1,0)}*-{\typetwo}="4"
!{(1,-1)}*{\typethree}="41"
"1"-"12"-"2"-"23"-"3"-"34"-"4"-"41"-"1"
"1"-|@{>>}"0"-|@{<<}"3" "2"-|@{<}"0"-"4"
"12"-"0"-"34" "23"-"0"-"41"
}
\ar[r]^>(.8){\sigma}
&
\mygraph{
!{<0cm,0cm>;<1cm,0cm>:<0cm,1cm>::}
!{(0,-1.5)}*-{\typetwo}="1"
!{(0,-1)}*-{\typethree}="2"
!{(0,0)}*-{\typetwo}="3"
!{(0,1)}*-{\typethree}="4"
!{(0,1.5)}*-{\typetwo}="5"
"1"-"2"-"3"-"4"-"5"
}
&
\mygraph{
!{<0cm,0cm>;<1cm,0cm>:<0cm,1cm>::}
!{(0,-1)}*-{\typethree}="12"
!{(0,0)}*-{\typetwo}="2"
!{(0,1)}*-{\typethree}="23"
"12"-"2"-"23"
}
&
\mygraph{
!{<0cm,0cm>;<1cm,0cm>:<0cm,1cm>::}
!{(1,-1)}*-{\typethree}="12"
!{(1,0)}*-{\typetwo}="2"
!{(1,1)}*-{\typethree}="23"
!{(-1,-1)}*-{\typethree}="15"
!{(-1,0)}*-{\typetwo}="5"
!{(-1,1)}*-{\typethree}="45"
!{(0,0)}*{\typeone}="0"
"12"-_>{v_2'}"2"-"23" "45"-"5"-_<{v_2}"15" "0"-|@{>}"2" "0"-|@{>}"5"
"0"-"12" "0"-"23" "0"-_<(.1){v_1}"45" "0"-"15"
} \quad \ar[l]_>(.8){\sigma}
\\
(a) &&& (b)
}
$$
\caption{}
\label{fig:impossible}
\end{figure}

\begin{proof}
In both cases we can assume that the seven central vertices are labelled as
in Figure~\ref{fig:ABPR}, with $f \in \K[x_2, x_3]$ a component, and $A, P, Q, R$ polynomials in
one variable (here we use Remark \ref{rem:same sigma} that characterises
vertices with the same image under $\sigma$).
\begin{figure}
$$\mygraph{
!{<0cm,0cm>;<1cm,0cm>:<0cm,1cm>::}
!{(1,1)}*-{\typethree}="12"
!{(1,0)}*-{\typetwo}="2"
!{(1,-1)}*-{\typethree}="23"
!{(-1,-1)}*-{\typethree}="45"
!{(-1,0)}*-{\typetwo}="5"
!{(-1,1)}*-{\typethree}="15"
!{(0,0)}*{\typeone}="0"
"12"-^<{[x_1 + A(f) + P(x_3),x_2 +Q(x_3),x_3]}^>{[x_2 +Q(x_3),x_3]}"2"-^>{[x_1 +
R(x_3),x_2 +Q(x_3),x_3]}"23"
"45"-^<{[x_1,x_2,x_3]}^>{[x_2,x_3]}"5"-^>{[x_1 + A(f),x_2,x_3]}"15"
"0"-"2" "0"-|@{>}"5"
"0"-"12" "0"-"23" "0"-"45" "0"-_<(0.2){[x_3]}"15"
}$$
\caption{}
\label{fig:ABPR}
\end{figure}

\begin{enumerate}[(a), wide]
\item
The existence of a vertex of type 2 between $[x_1 + A(f) + P(x_3),x_2
+Q(x_3),x_3]$ and $[x_1 + A(f),x_2,x_3]$, and the fact that $\deg Q \ge 2$
(otherwise the left and right type 2 vertices
would not be
distinct) imply $\deg P \le 1$.
So up to changing the representative we can assume $P = 0$, and by a symmetric
argument $R = 0$.
Then the right horizontal edge is oriented with respect to the leading term of
$A(f)$, as is the left horizontal edge, to the left.

\item
An outgoing right arrow would imply that there exist a polynomial $B
\in \K[X]$ and a component $g \in \K[x_2 + Q(x_3), x_3] = \K[x_2, x_3]$ such
that
\begin{equation}\label{eq:ABPR}
A(f) + B(g) + (P-R)(x_3) = 0.
\end{equation}
We can further assume that the first type 1 vertex on the path from $[x_3]$ to $[f]$ is $[x_2]$.
Then by Lemma \ref{lem:leading terms}(1,2)(and the fact that the vertices $[x_1,
x_2, x_3]$ and $[x_1 + A(f), x_2, x_3]$ are distinct, hence $\deg A(f) \ge 2$), there exist $d \ge 2, d' \ge 0$ such that the leading monomial of $A(f)$ is
$cx_2^{d}$, and the terms in $B(g)$ of degree greater than $d'$ have
the form $H(x_3)x_3^{d'} + c'x_2x_3^{d'}$ for some nonconstant $H$.

So either $d > d'$ and the monomial $cx_2^{d}$ from $A(f)$
cannot be canceled out by any other term in (\ref{eq:ABPR}), or $d'+1> d$ and the same remark applies to the monomial
$c'x_2x_3^{d'}$ from $B(g)$, leading to a contradiction in both cases.\qedhere
\end{enumerate}
\end{proof}

\begin{proof}[Proof of Proposition~\ref{prop:nonpositivecurvature}]
Since the link of a vertex $v_3$ of type 3 is the incidence graph of a projective plane, its minimal length immersed loop has length 6. Thus $K(v_3)\leq 0$ with the only equality case being Figure~\ref{fig:8cases}(a).

Consider now a vertex $v_2$ of type 2.
Recall that its link $\LL(v_2)$ is a (complete) bipartite graph.
If $\deg v_2=2n\geq 8$, then
$$K(v_2) = \frac{1}{4}(4 - \deg v_2) + \frac{\out v_2  - \inn v_2 }{6}\leq 1-2n\frac{1}{4}+n\frac{1}{6}=1-\frac{n}{3}<0.$$
In the case where $2n=6$, in order to obtain $K(v_2)<0$, we need to exclude the situation where there are three outgoing arrows from $v_2$.
This is exactly Lemma~\ref{lem:type2weird}. Finally, if $2n=4$, we have $K(v_2) = \frac{\out v_2  - \inn v_2 }{6}$.
By Lemma~\ref{lem:type2strange}, if there is an outgoing arrow from $v_2$, there is also an incoming arrow to $v_2$.
Thus $K(v_2)\leq 0$ with the only equality cases in Figure~\ref{fig:8cases}(b,c).

Finally consider a vertex $v_1$ of type $1$ with $\deg v_1=2n$. Let $\gamma$ be
the corresponding length $2n$ loop in $\LL(v_1)$. Since $\D$ is reduced,
$\gamma$ is embedded. For each pair of vertices of type 3 at distance 2 in
$\LL(v_1)$ there is only one length 2 path joining them. Hence $2n\geq 6$. We
have $K(v_1) = \frac{1}{6}(6 - n) + \frac{\out v_1  - \inn v_1 }{6}$. Denote by
$\non v_1 $ the number of nonoriented edges among the $n$ type $(1,2)$ edges
incident to~$v_1$. Since $n=\non v_1 +\out v_1 +\inn v_1 $, we have
$$K(v_1) = \frac{1}{6}\big(6 - \non v_1  -2\inn v_1 \big).$$

Consider the projection $\sigma\circ \gamma$ in the tree $\T(v_1)$ from Proposition~\ref{prop:maintree}.
Note that at each type 3 vertex in $\LL(v_1)$, the projection $\sigma$ is a local embedding.
If the number of vertices at which $\sigma\circ \gamma$ is not a local embedding (i.e.\ the number of backtracks in the tree $\T(v_1)$) is at least $3$, then by Lemma~\ref{lem:inarrow}, we have $\inn v_1 \geq 3$. Consequently $K(v_1)\leq 0$.
If $K(v_1)= 0$, then $\inn v_1 =3$, so $\sigma\circ \gamma$ is a tripod.
Moreover, since $\non v_1 =0$, by Lemma~\ref{lem:impossible}(b) all the tripod legs have length 1, so $n=3$ as in Figure~\ref{fig:8cases}(d).
It remains to consider the case where  $\sigma\circ \gamma$ is not a local embedding only at $2$ vertices.
In particular, its image is an embedded path of even length $n$, containing $\frac{n}{2}-1$ length $2$ paths to which we can apply Lemma~\ref{lem:impossible}(b).
If $K(v_1)\geq 0$, then $\non v_1 \leq 2$. Then by Lemma~\ref{lem:impossible}(b) we have $n\leq 6$.
Moreover, if $n=6$, then we are in one of the the cases of Figure~\ref{fig:8cases}(g,h).
It remains to consider the case where $n=4$.
To arrive at Figure~\ref{fig:8cases}(e,f), we need to prove that if there is an outgoing arrow from $v_1$, then the opposite edge of type (1,2) from $v_1$ is equipped with an incoming arrow.
This is exactly Lemma~\ref{lem:impossible}(a).
\end{proof}

As a consequence we now obtain the contractibility of the complex $\Comp$.

\begin{proof}[Proof of Theorem~\ref{thm:contractible}]
By \cite[Prop 5.7]{Lamy}, the complex $\Comp$ is simply connected. Since $\Comp$ is $2$-dimensional, by Whitehead and Hurewicz theorems,
it suffices to show that $\pi_2(\Comp)$ is trivial. Otherwise, using the same operation as in the proof of Lemma~\ref{lem:disc_exist}, we find a reduced sphere diagram in~$\Comp$. All its vertices are interior and by Proposition~\ref{prop:nonpositivecurvature} they all have nonpositive curvature. This contradicts Proposition~\ref{prop:GB}.
\end{proof}

\section{Hyperbolicity}
\label{sec:hyperbolicity}

In this section we prove Theorem~\ref{thm:hyperbolic}.
We will appeal to the following criterion for Gromov-hyperbolicity.

\begin{prop}
\label{prop:hypcriterion}
Let $X$ be a simplicial complex. Suppose that there exist constants $C,C',$
such that for each combinatorial loop $\gamma$ of length $l$ embedded in~$X$
there is a disc diagram $\D\rightarrow X$ with $\partial \D=\gamma$ and a set
$N\subset D^{0}$ of cardinality $\leq Cl$ such that the $C'$--neighbourhoods in
$\D^{1}$ around $\partial \D\cup N$ cover the entire $\D$. Then $X$ is
Gromov-hyperbolic.
\end{prop}

The proof is routine and we postpone it to the end of the section.

\begin{proof}[Proof of Theorem~$\ref{thm:hyperbolic}$]
For each combinatorial loop $\gamma$ of length $l$ embedded in $\Comp$, consider
a reduced disc diagram $\D\rightarrow \Comp$ with $\partial \D=\gamma$.
Let $N\subset \D^{0}$ be the set of interior vertices of $\D$ with negative
curvature.
Each vertex in $N$ has curvature $\leq -\frac{1}{6}$ by
Remark~\ref{rem:negativeminus1over6}.
Furthermore, for each boundary vertex of $\D$, its boundary curvature is $\le
\frac{1}{2}$ by Remark~\ref{rem:5over6}. Thus by Proposition~\ref{prop:GB}, we
have $|N|\leq 3l$, so we can take $C=3$. We will prove that $C'=11$ satisfies
the hypothesis of Proposition~\ref{prop:hypcriterion}.

\begin{figure}[t]
$$
\xymatrix{
\mygraph{
!{<0cm,0cm>;<3cm,0cm>:<0cm,1cm>::}
!{(1.5,0)}*{\typeone}="E"
!{(1,1)}*{\typeone}="N"
!{(.5,0)}*{\typeone}="W"
!{(1,-1)}*{\typeone}="S"
"E"-_<{[x_1+x_3^2]}_>{[x_2]}"N"-_>{[x_1]}"W"
"S"-|@{<<}"N"
"S"-_<{[x_3]}"E" "S"-"W"
}
\ar@{<~>}[r]
&
\mygraph{
!{<0cm,0cm>;<3cm,0cm>:<0cm,1cm>::}
!{(1.5,0)}*{\typeone}="E"
!{(1,1)}*{\typeone}="N"
!{(.5,0)}*{\typeone}="W"
!{(1,-1)}*{\typeone}="S"
!{(1,0)}*-{\typetwo}="O"
!{(.75,0.5)}*-{\typetwo}="NW"
!{(.75,-0.5)}*{\typetwo}="SW"
!{(1.25,0.5)}*-{\typetwo}="NE"
!{(1.25,-0.5)}*{\typetwo}="SE"
!{(.833,0)}*{\typethree}="3W"
!{(1.166,0)}*{\typethree}="3E"
"E"-_<{[x_1+x_3^2]}"NE"-_>{[x_2]}"N"-"NW"-_>{[x_1]}"W"
"S"-_<{[x_3]}"E" "S"-"W"
"N"-|@{>}"O"-|@{>>}"S" "E"-"O"-"W"
"N"-"3W"-"NW" "S"-"3W"-"SW"
"N"-"3E"-"NE" "S"-"3E"-"SE"
}
}
$$
\caption{Example of a directed edge in a simplified diagram with only type 1 vertices} \label{fig:arrows
type 1}
\end{figure}
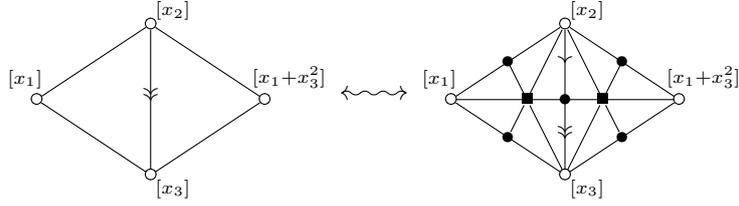

Let $Z\subset \D^{0}$ be the set of interior type 1 vertices with zero curvature.
By Proposition \ref{prop:nonpositivecurvature}, we have $Z = Z_6 \cup Z_8 \cup
Z_{12}$, where $Z_n$ denotes the subset of vertices in $Z$ with valence $n$.
We will prove that each vertex $v_1\in Z$ lies at distance $\leq 10$ from
$\partial \D\cup N$.
So assume by contradiction that there exists $v_1 \in Z$ with no vertex in
$\partial \D\cup N$ at distance $\leq 10$ from $v_1$.
Let $F\subset \D$ (for `Flat') be the union of all the triangles
of $\D$ with their vertices at distance $\leq 10$ in $\D^1$ from $v_1$. Then by Proposition~\ref{prop:nonpositivecurvature},
all interior vertices of $F$ of type $2$ and $3$ have valences $4$ and~$6$, respectively. In the figures of $F$ (or
subdiagrams of $F$) that
follow, we can thus combine the $6$ triangles of $F$ around each type $3$ vertex
to obtain a triangulation of $F$ with only type $1$ vertices. Here an arrow
between two vertices of type $1$ means two successive arrows, and a terminal
arrow means that the second arrow is terminal: see Figure~\ref{fig:arrows type
1}.
Observe also that in these simplified diagrams the valence of each type 1
vertex is half the original one.

\begin{figure}[t]
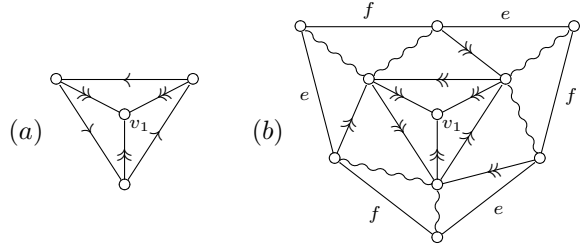

$$(a)\; \mygraph{
!{<0cm,0cm>;<.9cm,0cm>:<0cm,.7cm>::}
!{(-1,1)}*{\typeone}="NW"
!{(1,1)}*{\typeone}="NE"
!{(0,0.33)}*{\typeone}="O"
!{(0,-1)}*{\typeone}="S"
"NW"-|@{<}"NE"-|@{<}"S"-|@{<}"NW"
"NW"-|@{>>}"O" "NE"-|@{>>}^>{v_1}"O" "S"-|@{>>}"O"
}
\qquad
(b) \;\mygraph{
!{<0cm,0cm>;<.9cm,0cm>:<0cm,.7cm>::}
!{(-1,1)}*{\typeone}="NW"
!{(1,1)}*{\typeone}="NE"
!{(0,0.33)}*{\typeone}="O"
!{(0,-1)}*{\typeone}="S"
!{(-2,2)}*{\typeone}="_NW"
!{(0,2)}*{\typeone}="_N"
!{(2,2)}*{\typeone}="_NE"
!{(1.5,-0.5)}*{\typeone}="_E"
!{(0,-2)}*{\typeone}="_S"
!{(-1.5,-0.5)}*{\typeone}="_W"
"NW"-|@{<<}"NE"-|@{<<}"S"-|@{<<}"NW"
"NW"-|@{>>}"O" "NE"-|@{>>}^>{v_1}"O" "S"-|@{>>}"O"
"_NW"-^f"_N"-^e"_NE"-^f"_E"-^e"_S"-^f"_W"-^e"_NW"
"_NW"-@{~}"NW" "_NE"-@{~}"NE" "_S"-@{~}"S"
"NW"-@{~}"_N"-|@{>>}"NE"
"NE"-@{~}"_E"-|@{>>}"S"
"S"-@{~}"_W"-|@{>>}"NW"
}$$
\caption{Case $v_1 \in Z_6$} \label{fig:trianglecase}
\end{figure}

First consider the case where $v_1 \in Z_6$, corresponding to
Figure~\ref{fig:8cases}(d).
By Remark~\ref{rem:outgoing}, the boundary edges of the three triangles
containing $v_1$ must carry arrows, say counterclockwise as in
Figure~\ref{fig:trianglecase}(a) (the clockwise case is analogous).
Around the non-central type 1 vertices, we have two consecutive outgoing
arrows, so the only possible configuration is Figure~\ref{fig:8cases}(g), and
we obtain Figure~\ref{fig:trianglecase}(b) as a subset of the flat
neighbourhood $F$.
By Remark~\ref{rem:outgoing}, the edges labelled $e$ must be oriented
counterclockwise.
Then applying again Remark~\ref{rem:outgoing} to both endpoints of the edges labelled $f$ forces them to be oriented at the same time clockwise
and counterclockwise, contradiction.
Thus each $v_1 \in Z_6$ is at distance $\leq 4$ from $\partial \D\cup N$.

Consider now $v_1 \in Z_8$, that is, we are in one of the two cases of
Figure~\ref{fig:8cases}(e,f).

First consider case (f), see Figure ~\ref{fig:squarecase}.
Analysing the configuration at the top right vertex, by Remark~\ref{rem:outgoingbis}
we cannot simultaneously have the top edge oriented leftwards and the right edge oriented downwards.
Thus by Remark~\ref{rem:outgoing} applied to the top left and the bottom right vertex, without loss of generality the bottom edge is oriented rightwards, as in Figure~\ref{fig:squarecase}(a). Then with two consecutive outgoing arrows the bottom left vertex must correspond to
Figure~\ref{fig:8cases}(g), see Figure~\ref{fig:squarecase}(b).
Applying four times Remark~\ref{rem:outgoing}, we obtain
successively that the edges $e_1,e_2,e_3,e_4$ are all oriented clockwise.
Then Remark~\ref{rem:outgoingbis} gives a contradiction from the
point of view of the right bottom vertex.
Thus each $v_1 \in Z_8$ corresponding to case (f) is at distance $\leq 4$ from
$\partial \D\cup N$.

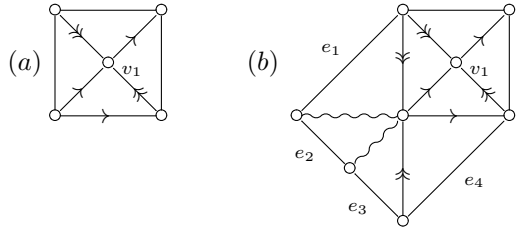
\begin{figure}[t]
$$
\xymatrix{
(a)\;\mygraph{
!{<0cm,0cm>;<.7cm,0cm>:<0cm,.7cm>::}
!{(-1,1)}*{\typeone}="NW"
!{(1,1)}*{\typeone}="NE"
!{(0,0)}*{\typeone}="O"
!{(-1,-1)}*{\typeone}="SW"
!{(1,-1)}*{\typeone}="SE"
"NW"-"NE"-"SE"-|@{<}"SW"-"NW"
"NW"-|@{>>}"O" "SE"-|@{>>}"O"
"NE"-|@{<}^>{v_1}"O"  "SW"-|@{>}"O"
}
&
(b)\;\mygraph{
!{<0cm,0cm>;<.7cm,0cm>:<0cm,.7cm>::}
!{(-1,1)}*{\typeone}="NW"
!{(1,1)}*{\typeone}="NE"
!{(0,0)}*{\typeone}="O"
!{(-1,-1)}*{\typeone}="SW"
!{(1,-1)}*{\typeone}="SE"
!{(-3,-1)}*{\typeone}="SW1"
!{(-2,-2)}*{\typeone}="SW2"
!{(-1,-3)}*{\typeone}="SW3"
"NW"-"NE"-"SE"-|@{<}"SW"-|@{<<}"NW"
"NW"-|@{>>}"O" "SE"-|@{>>}"O"
"NE"-|@{<}^>{v_1}"O"  "SW"-|@{>}"O"
"NW"-_{e_1}"SW1"-_{e_2}"SW2"-_{e_3}"SW3"-_{e_4}"SE"
"SW1"-@{~}"SW" "SW2"-@{~}"SW" "SW3"-|@{>>}"SW"
}
}
$$
\caption{Case $v_1 \in Z_8$, subcase (f)} \label{fig:squarecase}
\end{figure}

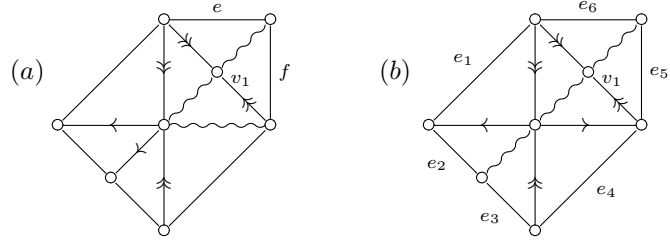
\begin{figure}[t]
$$
\xymatrix{
(a)\;\mygraph{
!{<0cm,0cm>;<.7cm,0cm>:<0cm,.7cm>::}
!{(-1,1)}*{\typeone}="NW"
!{(1,1)}*{\typeone}="NE"
!{(0,0)}*{\typeone}="O"
!{(-1,-1)}*{\typeone}="SW"
!{(1,-1)}*{\typeone}="SE"
!{(-3,-1)}*{\typeone}="SW1"
!{(-2,-2)}*{\typeone}="SW2"
!{(-1,-3)}*{\typeone}="SW3"
"NW"-^{e}"NE"-^{f}"SE"-@{~}"SW"-|@{<<}"NW"
"NW"-|@{>>}"O" "SE"-|@{>>}"O"
"NE"-@{~}^>{v_1}"O"  "SW"-@{~}"O"
"NW"-"SW1"-"SW2"-"SW3"-"SE"
"SW1"-|@{<}"SW" "SW2"-|@{<}"SW" "SW3"-|@{>>}"SW"
}
&
(b)\;\mygraph{
!{<0cm,0cm>;<.7cm,0cm>:<0cm,.7cm>::}
!{(-1,1)}*{\typeone}="NW"
!{(1,1)}*{\typeone}="NE"
!{(0,0)}*{\typeone}="O"
!{(-1,-1)}*{\typeone}="SW"
!{(1,-1)}*{\typeone}="SE"
!{(-3,-1)}*{\typeone}="SW1"
!{(-2,-2)}*{\typeone}="SW2"
!{(-1,-3)}*{\typeone}="SW3"
"NW"-^{e_6}"NE"-^{e_5}"SE"-|@{<}"SW"-|@{<<}"NW"
"NW"-|@{>>}"O" "SE"-|@{>>}"O"
"NE"-@{~}^>{v_1}"O"  "SW"-@{~}"O"
"NW"-_{e_1}"SW1"-_{e_2}"SW2"-_{e_3}"SW3"-_{e_4}"SE"
"SW1"-|@{<}"SW" "SW2"-@{~}"SW" "SW3"-|@{>>}"SW"
}
}
$$
\caption{Case $v_1 \in Z_8$, subcase (e)} \label{fig:squaretwocase}
\end{figure}

We continue with case (e). If the bottom left vertex is also of the form (e),
we replace $v_1$ with this vertex and observe that now the new bottom left
vertex is not of the form (e) anymore.
Indeed, otherwise the top left and bottom right vertices would have three consecutive
outgoing arrows contradicting Remark~\ref{rem:outgoing}.
So after at most one such replacement the bottom left vertex is of the form (g) or
(h).

On the one hand, if it is of the form (g), as in Figure~\ref{fig:squaretwocase}(a),
then by Remark~\ref{rem:outgoing}, the edges $e$ and $f$ are both oriented
away from the top right vertex. But this contradicts
Remark~\ref{rem:outgoing} at that vertex.

On the other hand, if the bottom
left vertex is of the form~(h), as in Figure~\ref{fig:squaretwocase}(b),
applying Remark~\ref{rem:outgoing} we obtain successively that $e_1$ is oriented clockwise, and $e_6,e_5$ are oriented counterclockwise.
Then the common vertex of $e_4$ and $e_5$ admits two consecutive outgoing arrows, hence must correspond to
case (g) of Figure \ref{fig:8cases}, and thus $e_4$ is nonoriented.
Applying Remark~\ref{rem:outgoing} again we get that $e_3,e_2$ are oriented counterclockwise.
By Remark~\ref{rem:outgoingbis}  this is not possible from the point of view of the common vertex of $e_1$ and $e_2$.

Given that we might have replaced $v_1$ at the beginning of the discussion, this
shows that each $v_1 \in Z_8$ in case (e) is
at distance $\leq 6$ from $\partial \D\cup N$.

Finally, consider $v_1 \in Z_{12}$.
We now prove that $v_1$ is at distance $\leq 4$ from
$Z_6 \cup Z_8  \cup \partial \D \cup N$.
Otherwise, the configurations at  all type 1 vertices at distance $\leq 4$ from
$v_1$ are as in Figure~\ref{fig:8cases}(g,h).
In particular up to symmetry we are in the situation of Figure
\ref{fig:hexagon}, where all type 1 vertices are of type (g) or (h), which
gives that all arrows in sight are terminal.
There are two possibilities for the two incoming arrows in
the dotted hexagon on the right, which lead to the two excluded diagrams of
the following final Lemma~\ref{lem:last}.

This proves that any type 1 vertex of $\D$ is at distance $\leq 4+6=10$ from
$\partial \D \cup N$ and consequently, any vertex of $\D$ is at distance $\leq
11$ from $\partial \D \cup N$.
\end{proof}

\begin{figure}[t]
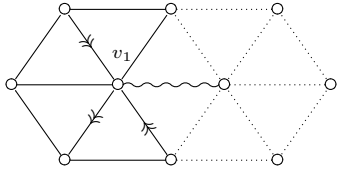

$$
\mygraph{
!{<0cm,0cm>;<.7cm,0cm>:<0cm,1cm>::}
!{(1,-1)}*{\typeone}="11"
!{(3,-1)}*{\typeone}="12"
!{(5,-1)}*{\typeone}="13"
!{(1,1)}*{\typeone}="31"
!{(3,1)}*{\typeone}="32"
!{(5,1)}*{\typeone}="33"
!{(0,0)}*{\typeone}="21"
!{(2,0)}*{\typeone}="22"
!{(4,0)}*{\typeone}="23"
!{(6,0)}*{\typeone}="24"
"21"-"11"-|@{<<}"22"-|@{<<}"12"-@{.}"23"-@{.}"13"-@{.}"24"
"21"-"31"-^>(.8){v_1}|@{>>}"22"-"32"-@{.}"23"-@{.}"33"-@{.}"24"
"11"-"12"-@{.}"13"
"21"-"22"-@{~}"23"-@{.}"24"
"31"-"32"-@{.}"33"
}$$
\caption{Case $v_1 \in Z_{12}$}
\label{fig:hexagon}
\end{figure}

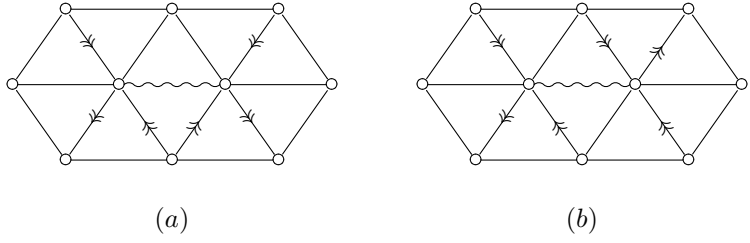
\begin{figure}[t]
$$
\xymatrix@R-1pc{
\mygraph{
!{<0cm,0cm>;<.7cm,0cm>:<0cm,1cm>::}
!{(1,-1)}*{\typeone}="11"
!{(3,-1)}*{\typeone}="12"
!{(5,-1)}*{\typeone}="13"
!{(1,1)}*{\typeone}="31"
!{(3,1)}*{\typeone}="32"
!{(5,1)}*{\typeone}="33"
!{(0,0)}*{\typeone}="21"
!{(2,0)}*{\typeone}="22"
!{(4,0)}*{\typeone}="23"
!{(6,0)}*{\typeone}="24"
"21"-"11"-|@{<<}"22"-|@{<<}"12"-|@{>>}"23"-|@{>>}"13"-"24"
"21"-"31"-|@{>>}"22"-"32"-"23"-|@{<<}"33"-"24"
"11"-"12"-"13"
"21"-"22"-@{~}"23"-"24"
"31"-"32"-"33"
}
&
\mygraph{
!{<0cm,0cm>;<.7cm,0cm>:<0cm,1cm>::}
!{(1,-1)}*{\typeone}="11"
!{(3,-1)}*{\typeone}="12"
!{(5,-1)}*{\typeone}="13"
!{(1,1)}*{\typeone}="31"
!{(3,1)}*{\typeone}="32"
!{(5,1)}*{\typeone}="33"
!{(0,0)}*{\typeone}="21"
!{(2,0)}*{\typeone}="22"
!{(4,0)}*{\typeone}="23"
!{(6,0)}*{\typeone}="24"
"21"-"11"-|@{<<}"22"-|@{<<}"12"-"23"-|@{<<}"13"-"24"
"21"-"31"-|@{>>}"22"-"32"-|@{>>}"23"-|@{>>}"33"-"24"
"11"-"12"-"13"
"21"-"22"-@{~}"23"-"24"
"31"-"32"-"33"
}
\\
(a) & (b)
}$$
\caption{Forbidden bihexagonal diagrams}
\label{fig:last}
\end{figure}

\begin{lemma}
\label{lem:last}
There are no reduced disc diagrams as in Figure~\ref{fig:last}, (a) or (b).
\end{lemma}

\begin{proof}
For Figure~\ref{fig:last}(a), using the action of the group $\TA$, we can make the choice of representatives as in Figure~\ref{fig:bihex}, where $A, B, C, D \in \K[X]$ are polynomials of degree at least 2.
\begin{figure}
$$
\mygraph{
!{<0cm,0cm>;<.7cm,0cm>:<0cm,1cm>::}
!{(1,-1)}*{\typeone}="11"
!{(3,-1)}*{\typeone}="12"
!{(5,-1)}*{\typeone}="13"
!{(1,1)}*{\typeone}="31"
!{(3,1)}*{\typeone}="32"
!{(5,1)}*{\typeone}="33"
!{(0,0)}*{\typeone}="21"
!{(2,0)}*{\typeone}="22"
!{(4,0)}*{\typeone}="23"
!{(6,0)}*{\typeone}="24"
"21"-"11"-|@{<<}"22"-|@{<<}"12"-|@{>>}"23"-|@{>>}"13"-"24"
"21"-"31"-|@{>>}"22"-"32"-"23"-|@{<<}"33"-"24"
"11"-_<(-.2){x_2 + D(x_3)}_>{?}"12"-_>{x_3 +
B(x_2)}"13"
"21"-^<(-.5){x_1 + C(x_3)}"22"-^<(.2){x_3}^>(.8){x_2}"23"-^>(1.5){x_1 +
A(x_2)}"24"
"31"-"32"-^<{x_1}"33"
}$$
\caption{}
\label{fig:bihex}
\end{figure}
Moreover, there exist polynomials $E, F \in \K[X]$, also of degree at least 2, such that the central
bottom vertex has the form
$$[x_1+A(x_2) + E(x_3+B(x_2))] = [x_1+C(x_3) + F(x_2+D(x_3))].$$
By computing the second-order derivative $\frac{\partial^2 }{\partial x_2\partial x_3}$ on both sides we obtain
$$B'(x_2)E''(x_3+B(x_2)) =  D'(x_3)F''(x_2+D(x_3)),$$
which is impossible in view of the monomials of highest degrees.

Similarly in the situation of Figure~\ref{fig:last}(b) we can choose representatives as in Figure~\ref{fig:bihex2}.
\begin{figure}
$$
\mygraph{
!{<0cm,0cm>;<.7cm,0cm>:<0cm,1cm>::}
!{(1,-1)}*{\typeone}="11"
!{(3,-1)}*{\typeone}="12"
!{(5,-1)}*{\typeone}="13"
!{(1,1)}*{\typeone}="31"
!{(3,1)}*{\typeone}="32"
!{(5,1)}*{\typeone}="33"
!{(0,0)}*{\typeone}="21"
!{(2,0)}*{\typeone}="22"
!{(4,0)}*{\typeone}="23"
!{(6,0)}*{\typeone}="24"
"21"-"11"-|@{<<}"22"-|@{<<}"12"-"23"-|@{<<}"13"-"24"
"21"-"31"-|@{>>}"22"-"32"-|@{>>}"23"-|@{>>}"33"-"24"
"11"-_<(-.2){x_2 + D(x_3)}_>{?}"12"-"13"
"21"-^<(-.5){x_1 + C(x_3)}"22"-^<(.2){x_3}^>(.8){x_2}"23"-^>(1.9){x_1 +
B(x_3+A(x_2))}"24"
"31"-"32"-^<{x_1}^>{x_3 + A(x_2)}"33"
}$$
\caption{}
\label{fig:bihex2}
\end{figure}
Then there exist polynomials $E, F \in \K[X]$ such that the central
bottom vertex has the form
$$[x_1+E(x_2) + B(x_3+A(x_2))] = [x_1+C(x_3) + F(x_2+D(x_3))],$$
and we get a contradiction as in the previous case.
\end{proof}

\begin{proof}[Proof of Proposition~\ref{prop:hypcriterion}]
By \cite[Thm~III.H.2.9]{BH}, via Definitions 2.1, and Remarks 2.3(1,2,5) therein, to prove that $\Comp$ is Gromov-hyperbolic it suffices to show that any combinatorial loop $\gamma$ embedded in $\Comp$ is the boundary loop of a following triangulated disc $T\rightarrow \Comp$.
First of all, we require that $|T^0|$ is linear in $|\gamma|$.
Secondly, we require that adjacent vertices in $T^0$ are sent to vertices in $\Comp^0$ at a uniformly bounded distance.
The map $T\rightarrow \Comp$ does not need to be combinatorial (or even continuous).

To find such $T$, let $\D$ be the disc diagram with boundary $\gamma$ guaranteed
by the hypothesis. For each vertex $v$ of $\D$ outside $N\cup \partial \D^0$
choose an edge $e(v)$ joining $v$ to a vertex closer to $N\cup \partial \D^0$ in
$\D^1$. Combine the triangles of the barycentric subdivision $\D'$ of $\D$
incident to a common vertex of $\D$ to form a combinatorial subdivision $S$ of
$\D$. Combine the cells of $S$ containing vertices joined by some $e(v)$. Label
each such combined cell $s$ by the unique vertex $v(s)\in N\cup \partial \D^0$
that it contains. By the hypothesis, any vertex in such a cell is at distance
$\leq C'$ in $\D^1$ from $v(s)$. If some cell $s$ becomes non-simply connected,
incorporate all the cells that it bounds into $s$. We thus obtain a
combinatorial subdivision $S'$ of $\D$ with at most $|N\cup \partial \D^{0}|\leq
Cl$ cells. Since at each vertex of $S'$ there are exactly 3 cells meeting,
consider the triangulation $T$ dual to $S'$ with vertices labelled $v(s)$. For
each edge $v(s)v(s')$ of $T$, in the metric of $\D^1$ we have $|v(s),v(s')|\leq
2C'+1$, as desired.
\end{proof}

\section{Loxodromic WPD element}
\label{sec:loxodromic}

In this section we prove Theorems~\ref{thm:loxo} and~\ref{thm:wpd}. Recall
that $g,h,f\in\TA$ are of the form $g^{-1}=(x_2,x_1+x_2x_3,x_3), \
h^{-1}=(x_3,x_1,x_2),\ f=g^n\circ h$, where $n\geq 0$. We now define a candidate axis $\gamma$ for $f$.

\begin{defin} Consider the following map $\gamma\colon \R \rightarrow
\Comp^1$. Let $\gamma(0)=[x_1],\gamma(1)=[x_1,x_3],\gamma(2)=[x_3]$, and
let $\gamma[0,2]$ be the corresponding length 2 path.
Note that $f \cdot \gamma(0) =[x_1\circ f^{-1}]=[x_1\circ h^{-1} \circ
g^{-n}]=[x_3]=\gamma(2)$. We can thus extend the definition to
$\gamma[2k, 2k+2]=f^k \cdot \gamma[0,2]$.
\end{defin}

\begin{rem}
\label{rem:easyboundarycurv}
For any reduced disc diagram with $\gamma[0,2]$ in the boundary, the boundary
curvature at the type $2$ vertex $\gamma(1)$ is nonpositive, since
there are at least 3 nonoriented edges from $\gamma(1)$: two boundary edges
and one of type $(2,3)$.
\end{rem}

We shall need the following technical lemma, which is similar to Lemma \ref{lem:impossible}.

\begin{lemma} \label{lem:no outgoing}
In the situation of Figure \ref{fig:no outgoing}, the horizontal edge on the
right is nonoriented.
\end{lemma}

\begin{figure}
$$
\xymatrix@R-1pc{
\mygraph{
!{<0cm,0cm>;<1cm,0cm>:<0cm,1cm>::}
!{(0,-1)}*-{\typethree}="12"
!{(0,0)}*-{\typetwo}="2"
!{(0,1)}*-{\typethree}="23"
"12"-"2"-"23"
}
&
\mygraph{
!{<0cm,0cm>;<1cm,0cm>:<0cm,1cm>::}
!{(1,-1)}*-{\typethree}="12"
!{(1,0)}*-{\typetwo}="2"
!{(1,1)}*-{\typethree}="23"
!{(-1,-1)}*-{\typethree}="15"
!{(-1,0)}*-{\typetwo}="5"
!{(-1,1)}*-{\typethree}="45"
!{(0,0)}*{\typeone}="0"
"12"-_<{[x_1 +R(x_3),x_2+Q(x_3), x_3]}_>{[x_2+Q(x_3), x_3]}"2"-_>{[x_1
+ x_2x_3 +P(x_3),x_2+Q(x_3), x_3]}"23"
"45"-_<{[x_1 + x_2x_3,x_2, x_3]}"5"-_<{[x_2,x_3]}_>{[x_1 ,x_2, x_3]}"15"
"0"-"2" "0"-@{~}"5"
"0"-"12" "0"-"23" "0"-_<(.2){[x_3]}"45" "0"-"15"
} \quad \ar[l]_>(.9){\sigma}
}
$$
\caption{Here $P,Q,R \in \K[X]$ are arbitrary polynomials, so some
of the vertices on the right and on the left might coincide}
\label{fig:no outgoing}
\end{figure}
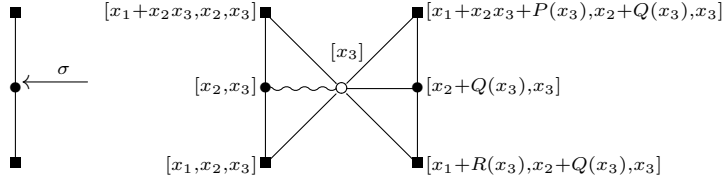

\begin{proof}
Assume the right horizontal edge is oriented.
Then there exist a component $\phi(x_2, x_3) \in \K[x_2, x_3] = \K[x_2 + Q(x_3),x_3]$ and a polynomial $A(X) \in \K[X]$ such that
$$x_2x_3 + (P-R)(x_3) + A(\phi) = 0.$$
Comparing degrees in $x_2$ we must have
$\deg_{x_2} \phi = 1$ and $\deg A = 1$, so that $A(\phi)$ has the form
$$A(\phi) = ax_2 + S(x_3).$$
But then $A(\phi)$ contains no terms cancelling out with $x_2x_3$, contradiction.
\end{proof}

We now establish the key result towards the proof that $f = g^n \circ h$ is a
loxodromic WPD element, which is an analogue of
Remark~\ref{rem:easyboundarycurv} for a type $1$ vertex.

\begin{prop}
\label{pro:boundarycurvature}
Let $\D\rightarrow \Comp$ be any reduced disc diagram that agrees on a subpath of
$\partial \D$ with $\gamma[1,3]$. Then the boundary curvature at
$\gamma(2)$ is
$\leq \frac{2-n}{6}$.
\end{prop}

\begin{proof}
Denote by in, out, and non the number of type (1,2) edges incident to
$\gamma(2)$ in $\D$ that carry incoming, outgoing, or no arrows,
respectively. Since $\gamma(1),\gamma(3)$ are both of type $2$, the
valence of $\gamma(2)$ in $\D$ equals
$2(\mathrm{in}+\mathrm{out}+\mathrm{non})-1$.
Thus the boundary curvature at $\gamma(2)$ is
$$K_{\partial}(\gamma(2)) = \frac{1}{12}\Big(7-2(\mathrm{in}+\mathrm{out}+\mathrm{non})+1\Big)+\frac{\mathrm{out}-\mathrm{in}}{6}=\frac{4-2\mathrm{in}-\mathrm{non}}{6}.$$
The two boundary edges do not carry arrows by definition. Thus to prove
the proposition it suffices to find $n$ other edges, either nonoriented or
with an incoming arrow (in the following estimate, we do not try to take
advantage of the fact that the incoming arrows have a coefficient 2 in the above
formula).

Observe that
$$\gamma(3)=f \cdot \gamma(1)= (g^{n} \circ h)\cdot [x_1, x_3] =  g^{n}\cdot [x_2, x_3].$$
In the link $\mathcal{L}([x_3])$ of $[x_3]$, let $e_0^-$ be the edge from $[x_1,x_3]$ to
$[x_1,x_2,x_3]$ and let $e_0^+$ be the edge from $[x_1,x_2,x_3]$ to
$[x_2,x_3]$.
For $k=1,\ldots, n,$ let $e_k^-=g^k \cdot e_0^-$ and $e_k^+=g^k \cdot e_0^+$.
Since $g \cdot [x_1,x_3]=[x_2,x_3]$, we have that
$\delta=e_0^-e_0^+e_1^-e_1^+\cdots e_n^-e_n^+$ is a path from $\gamma(1)$
to $\gamma(3)$.

We will now use the projection $\sigma\colon \mathcal{L}([x_3])\to
\mathcal{T}([x_3])$ from Proposition~\ref{prop:maintree}.
Since $e_1^-$ ends at $g \cdot [x_1, x_2, x_3]  = [x_1+x_2x_3,x_2,x_3]$, we have that
$\sigma(e_0^-e_0^+e_1^-)$ is an
embedded length 3 path in the tree $\mathcal{T}([x_3])$. Consequently
$\sigma\circ\delta$ is a geodesic.

Let $\hat{\delta}$ be the link of $\gamma(2)$ in $\D$.
It is an other path in $\mathcal{L}([x_3])$ from $\gamma(1)$ to $\gamma(3)$.
Thus the path
$\sigma\circ \hat{\delta}$ in $\mathcal{T}([x_3])$ has the same endpoints as the geodesic
$\sigma\circ\delta$.
Consequently, there are edges
$\hat{e}_0^-,\hat{e}_0^+,\hat{e}_1^-,\ldots,\hat{e}_n^+$
that appear in $\hat{\delta}$ in that order satisfying
$\sigma(\hat{e}_k^\pm)=\sigma(e_k^\pm)$ for all $k = 0, \dots, n$.

Let $\tilde{\delta}\subset \hat{\delta}$ be the subpath between the
endpoint of $\hat{e}_0^+$ and the starting point of~$\hat{e}_1^-$. If
$\tilde{\delta}$ is not a single vertex, then $\sigma$ maps it to a closed
path in $\mathcal{T}([x_3])$. Thus $\sigma\circ \tilde{\delta}$ is not a
local embedding, and by Lemma~\ref{lem:inarrow} there exists a vertex $v_2$ in
$\tilde{\delta}$ with a terminal arrow from $v_2$ to $\gamma(2)$ in $\D$.
If $\tilde{\delta}$ is a single vertex $v_2$, then we obtain the
configuration of Figure \ref{fig:no outgoing}, where $e_0^+, e_1^-$ are the
vertical edges on the left, and $\hat e_0^+, \hat e_1^-$ are the vertical
edges on the right.
By Lemma \ref{lem:no outgoing}, the edge from $v_2$ to
$\gamma(2)$ is nonoriented.

Analogously one proves that for each $k=1,\ldots, n-1$, there is a vertex
$v_2$ between the endpoint of $\hat{e}_k^+$ and the starting point of
$\hat{e}_{k+1}^-$, with a terminal arrow or nonoriented edge from $v_2$ to
$\gamma(2)$.
\end{proof}

\begin{proof}[Proof of Theorem~\ref{thm:loxo}]
We will prove that for any $s<t\in\Z$, the distance between $\gamma(s)$ and $\gamma(t)$ is $>\lambda |t-s|$, where $\lambda=\frac{1}{6}$.
Otherwise, let $s,t$ be closest such that the above estimate is violated, and let $\alpha$ be a geodesic in $\Comp^1$ from $\gamma(t)$ to $\gamma(s)$.
We claim that the loop $\gamma[s,t]\alpha$ is embedded.
Indeed, if already $\gamma[s,t]$ self-intersected, then the maximal embedded interval $[s',t']\subset [s,t]$ would also violate our estimate.
Moreover, if $\alpha$ intersected $\gamma[s,t]$ in a point $\gamma(t')$, then either $[s,t']$ or $[t',t]$ would violate our estimate. This justifies the claim that $\gamma[s,t]\alpha$ is embedded.

Let $\D$ be a reduced disc diagram with boundary $\gamma[s,t]\alpha$. By
Proposition~\ref{prop:nonpositivecurvature}, all the interior vertices of $\D$
have nonpositive curvature. By Remark~\ref{rem:5over6}, each vertex of $\alpha$,
including the endpoints,
has boundary curvature $\leq \frac{1}{2}$, so the sum of their boundary curvatures is $\leq \frac{1}{2}(\lambda(t-s)+1)$. By Proposition~\ref{pro:boundarycurvature},
for every even integer $s<k<t$, the boundary curvature at $\gamma(k)$ is $\leq -\frac{1}{6}$, and by Remark~\ref{rem:easyboundarycurv}, the boundary curvature is
nonpositive at odd integers. By Proposition~\ref{prop:GB}, we obtain $$1\leq \frac{1}{2}(\lambda(t-s)+1)-\frac{1}{6}\Big(\frac{t-s}{2}-1\Big).$$ After comparing the
terms with and without $(t-s)$ this yields a contradiction.
\end{proof}

Along the proof we established the following.

\begin{cor}
$\gamma$ is embedded.
\end{cor}

For the proof of Theorem~\ref{thm:wpd}, we need some preparatory lemmas.

\begin{lemma}
\label{lem:uniquegeodesic}
Let $n\geq 12$ and let $\alpha$ be any geodesic in $\Comp^1$ with endpoints $\gamma(t),\gamma(s)$, where $s<t$, and with other vertices disjoint from $\gamma$. Then $t-s=2$ and $s,t$ are even.
\end{lemma}
\begin{proof}
Let $\D$ be a reduced disc diagram with boundary $\gamma[s,t]\alpha$. By
Proposition~\ref{prop:nonpositivecurvature}, all the interior vertices of $\D$
have nonpositive curvature. By Remark~\ref{rem:5over6}, the sum of the boundary
curvatures of the
vertices of $\alpha$ is $\leq \frac{1}{2}(t-s+1)=(\frac{t-s}{2}-2)+\frac{5}{2}$. By Proposition~\ref{pro:boundarycurvature}, for every even integer $s<k<t$, the boundary curvature at $\gamma(k)$ is $\leq \frac{2-n}{6}< -\frac{3}{2}$.
Assuming, by contradiction, that $t-s>2$, or $t-s=2$ and $t,s$ are odd, we obtain that there exists such $k$.
Together with Remark~\ref{rem:easyboundarycurv}, this shows that the sum of the boundary curvatures at the interior vertices of $\gamma[s,t]$ is
$<-\frac{3}{2}$.

Moreover, if $t-s\geq 4$, so $\frac{t-s}{2}-2\geq 0$, then the sum of the boundary curvatures at the interior vertices of $\gamma[s,t]$ is $<-\frac{3}{2}-\frac{3}{2}(\frac{t-s}{2}-2)$.
In that case by Proposition~\ref{prop:GB}, we obtain $1< (\frac{t-s}{2}-2)+\frac{5}{2}-\frac{3}{2}-\frac{3}{2}(\frac{t-s}{2}-2)\leq 1$, which is a contradiction.
In the case where $t-s=3$, or where $t-s=2$ and $t,s$ are odd, from Proposition~\ref{prop:GB} we obtain $1< \frac{1}{2}\cdot 4 -\frac{3}{2}$, which is again a contradiction.
\end{proof}

We have the following immediate consequences.

\begin{cor}
\label{cor:uniquegeodesic}
Let $n\geq 12$ and let $\alpha$ be any geodesic with endpoints $\gamma(t),\gamma(s)$. Then for any even $k$ between $s$ and $t$, the vertex $\gamma(k)$ lies in $\alpha$.
\end{cor}

\begin{cor}
\label{cor:geodesic}
For $n\geq 12$, $\gamma$ is a geodesic.
\end{cor}

We continue with one more similar lemma.

\begin{lemma}
\label{lem:gammaattracting}
Let $n\geq 12$ and $C\geq 1$. Suppose that $v,v'\in \Comp^{0}$ are at
distance $\leq C$ from $\gamma(s),\gamma(t)$, respectively. Moreover, assume $|s-t|\geq 12C$.
Then any geodesic from $v'$ to $v$ intersects $\gamma$.
\end{lemma}

\begin{proof}
We prove the lemma by contradiction. Suppose that $v',v$ are as above with
geodesic $\alpha$ between them disjoint from $\gamma$. Moreover, assume
that the distance between $v'$ and $v$ is minimal possible. We can assume $s< t$. Let $\beta, \beta'$ be geodesics between $v$ and $\gamma(s)$, and $\gamma(t)$ and $v'$. Then the loop
$\delta=\gamma[s,t]\beta'\alpha\beta$ is embedded, since otherwise we
could replace one of $v,v'$ by a vertex of self-intersection, decreasing the
distance between $v'$ and $v$.

Let $\D\rightarrow \Comp$ be a reduced disc diagram with boundary $\delta$.
By Proposition~\ref{prop:nonpositivecurvature}, all the interior vertices
of $\D$ have nonpositive curvature. Since $\alpha$ is a geodesic, we have
$|\alpha|\leq 2C+t-s$. Hence by Remark~\ref{rem:5over6}, the sum of the boundary curvatures at the
vertices of $\beta'\alpha\beta$ is $\leq \frac{1}{2} (4C+t-s+1)$.
By Remark~\ref{rem:easyboundarycurv} and Proposition~\ref{pro:boundarycurvature}, the sum of the boundary curvatures at the interior vertices of $\gamma[s,t]$ is $\leq \frac{2-n}{6}\big(\frac{t-s}{2}-1\big)< -\frac{3}{4}\big((t-s)-2\big)$. By Proposition~\ref{prop:GB}, we obtain
\begin{multline*}
1< \frac{1}{2}\Big(1+4C+(t-s)\Big)-\frac{3}{4}\Big((t-s)-2\Big) \\
= \frac{1}{2}+2C-\frac{1}{4}\big(t-s\big)+\frac{3}{2}\leq 2C-\frac{1}{4}12C+2\leq 2-1,
 \end{multline*}
which is a contradiction.
\end{proof}

Recall that we defined $g^{-1} = (x_2,x_1+x_2x_3,x_3)$.
We now compute the iterates of $g$.

\begin{lemma}\label{lem:iterates of g}
There exist polynomials $P_n, Q_n \in \K[x_3]$ such that for any $n \ge 1$
\begin{align*}
g^n &= (x_1 Q_{n}(x_3) + x_2 Q_{n-1}(x_3), x_1 Q_{n-1}(x_3) + x_2
Q_{n-2}(x_3), x_3)\\
g^{-n} &=  (x_1 P_{n-2}(x_3) + x_2 P_{n-1}(x_3), x_1 P_{n-1}(x_3) + x_2
P_{n}(x_3), x_3).
\end{align*}
Moreover $P_n, Q_n$ is an odd (resp.\ even) polynomial if $n$ is odd (resp.\ even),
and $P_n(0) = Q_n(0) = 1$ for any even $n$.
\end{lemma}

\begin{proof}
The existence of the polynomials follows from the fact that we can view $g^{\pm
1}$ as elements of $\GL_2(\K[x_3])$:
$$g = \begin{pmatrix}
      -x_3 & 1 \\ 1 & 0
      \end{pmatrix}
,\quad
g^{-1} = \begin{pmatrix}
      0 & 1 \\ 1 & x_3
      \end{pmatrix}.
$$
It is natural to define the following sequences in $\K[x_3]$:
\begin{align*}
P_{-1} = 0, P_0 = 1, \text{ and for all } n \ge 1, P_n := x_3 P_{n-1} +
P_{n-2}, \\
Q_{-1} = 0, Q_0 = 1, \text{ and for all } n \ge 1, Q_n := -x_3 Q_{n-1} +
Q_{n-2}.
\end{align*}
Then we have
$$g^{n} = \begin{pmatrix}
      Q_{n} & Q_{n-1} \\ Q_{n-1} & Q_{n-2}
      \end{pmatrix}
, \quad
g^{-n} = \begin{pmatrix}
      P_{n-2} & P_{n-1} \\ P_{n-1} & P_n
      \end{pmatrix}
$$
and the stated properties follow by induction.
\end{proof}

\begin{lemma}
\label{lem:axisgivesinfo}
Assume $n \geq 2$ be even.
For $m-l \ge 7$, there are only finitely many elements in $\TA$
fixing the segment $\gamma[l,m]$.
\end{lemma}

\begin{proof}
Since $f$ acts as a translation of length 2 on $\gamma$, it is sufficient to
prove that there are only finitely many elements in $\TA$
fixing the segment $\gamma[-2,4]$.
So assume that $\phi^{-1} = (\phi_1,
\phi_2, \phi_3)$ fixes $\gamma[-2,4]$.
First, the fact that $\phi^{-1}$ fixes $\gamma(0) = [x_1]$ and $\gamma(2) =
[x_3]$ means that
$$ \phi_1 = ax_1 + b, \quad \phi_3 = cx_3 + d,$$
for some $a,c \in \K^*$, $b,d \in \K$.
This implies that $\phi_2 = e x_2 + Q(x_1, x_3)$ for some $e \in \K^*$ and $Q
\in \K[x_1, x_3]$.
Now, using the notation from Lemma \ref{lem:iterates of g}, we consider the action on
$$
\gamma(4) = f\cdot[x_3] = [x_3 \circ h^{-1} \circ g^{-n}] = [x_2 \circ g^{-n}]
= [x_1P_{n-1}(x_3) + x_2P_n(x_3)].$$
We have
$$\phi \cdot \gamma(4) = [(ax_1+b)P_{n-1}(cx_3+d) +
(ex_2+Q(x_1,x_3))P_n(cx_3+d)],$$
and the condition $\phi \cdot \gamma(4) = \gamma(4)$ means that there exist
$\alpha, \beta$ such that
\begin{multline*}
(ax_1+b)P_{n-1}(cx_3+d) + (ex_2+Q(x_1,x_3))P_n(cx_3+d) \\
= \alpha x_1 P_{n-1}(x_3)+ \alpha x_2 P_n(x_3) + \beta.
\end{multline*}
Since these polynomials are linear in $x_2$, we get in fact two equations
\begin{align*}
&(ax_1+b)P_{n-1}(cx_3+d) + Q(x_1,x_3) P_n(cx_3+d)  = \alpha x_1P_{n-1}(x_3) +
\beta,\\
&eP_n(cx_3+d) = \alpha P_n(x_3).
\end{align*}
Comparing the terms with factor $x_3^n$ of the first equation gives $Q = 0$. By Lemma~\ref{lem:iterates of g}, $P_n$ is even and hence on the right side of the second equation we have no $x_3^{n-1}$ term, and consequently $d=0$.
Thus our system of equations is now
\begin{align*}
&(ax_1+b)P_{n-1}(cx_3)  = \alpha x_1 P_{n-1}(x_3) + \beta,\\
&eP_n(cx_3) = \alpha P_n(x_3).
\end{align*}
By Lemma \ref{lem:iterates of g} we have $P_n(0) = 1$ and
$P_{n-1}(0) = 0$, so that we get successively
$$\alpha = e, \beta = 0,  c = \pm 1, b = 0, \alpha = ca.$$
Finally, using the notation from Lemma \ref{lem:iterates of g}, we have
$$\gamma(-2) = f^{-1}\cdot[x_1] = [x_1 \circ g^n \circ h]
= [Q_{n}(x_1)x_2 + Q_{n-1}(x_1)x_3].$$
From the constraint $\phi\cdot \gamma(-2) = \gamma(-2)$, we have $\alpha', \beta'$ such that
$$Q_{n}(ax_1)ex_2 + Q_{n-1}(ax_1)cx_3 = \alpha' Q_{n}(x_1)x_2 + \alpha' Q_{n-1}(x_1)x_3  + \beta'.$$
Comparing the terms with factor $x_2$ we get
$$eQ_{n}(ax_1) = \alpha'Q_{n}(x_1).$$
Since $Q_{n}(0) = 1$ we get $e = \alpha'$, and then $a = \pm1$.
Summarising, $\phi$ is in the diagonal group of order four  $\{(ax_1, acx_2, cx_3); a = \pm1, c = \pm1\}$.
\end{proof}

\begin{proof}[Proof of Theorem~\ref{thm:wpd}]
For any $C\geq 1$, let $2k\geq 25C+6$. Let $x=\gamma(0)$.
Suppose that we have $j\in \TA$ with $v=j\cdot x,\ v'=j\circ f^k\cdot x$ at
distance $\leq C$ from $x, f^k\cdot x=\gamma(2k)$, respectively. By Corollary~\ref{cor:geodesic}, we have that $\alpha=j\cdot \gamma[0,2k]$ is a geodesic from~$v$ to~$v'$.

We will now prove that $\alpha\cap \gamma[0,2k]$ has the form $\gamma[l,l']$. By Lemma~\ref{lem:gammaattracting}, $\alpha$ intersects $\gamma[0,2k]$.
Let $l,l'\in [0,2k]$ be minimal, and maximal, respectively, with $\gamma(l),\gamma(l')\in \alpha$.
By Corollary~\ref{cor:uniquegeodesic},
for any even $l\leq k'\leq l'$, the vertex $\gamma(k')$ lies in~$\alpha$.
In fact, since any pair of type $1$ vertices has at most one common neighbour of type~$2$, the entire $\gamma[l,l']$ lies in $\alpha$.

Let $m$ be such that $\gamma(l)=j\cdot\gamma(m)$. We first exclude the possibility that $\alpha$ and $\gamma[0,2k]$ have reverse orientations on the common part, meaning $\gamma(l')=j\cdot \gamma(m-(l'-l))$. Indeed, then by Lemma~\ref{lem:gammaattracting} applied to $j\cdot\gamma[m+1,2k]$, we obtain $2k-l<12C$. Analogously, we get $l'<12C$ and consequently $2k<24C$, which is a contradiction. Thus we have $\gamma(l')=j\cdot \gamma(m+(l'-l))$.

By Lemma~\ref{lem:gammaattracting} applied to $j\cdot\gamma[0,m-1]$, we obtain $l<12C$. Analogously, $2k-l'<12C$. Assume without loss of generality $m\geq l$. Then the map $f^{{(m-l)}/{2}}\circ j$ fixes the entire $\gamma[m,l']$.
By the triangle inequality, we have $m-l\leq C$.
Consequently, $l'-m\geq l'-l-C> 2k-12C-12C-C=2k-25C\geq 6$. By Lemma~\ref{lem:axisgivesinfo}, there are only finitely many such $j$. This shows that $f$ is WPD.
\end{proof}

\begin{rem} The complex $\mathcal{C}$ does not satisfy an isoperimetric inequality in the following sense. Consider an infinite geodesic path in $\mathcal{T}([x_1,x_2])$ starting at $v^0=[x_1]$ and passing successively through vertices $v^1,v^2,\ldots$ of type $1$. Let $\alpha_n$ be the length $8$ loop in $\mathcal{C}$  passing through type 1 vertices $v^0,[x_3],v^n,[x_3+x_1^2]$. Since the two neighbours of $[x_3]$ on $\alpha_n$ are at distance $2n$ in $\mathcal{L}([x_3])$, the minimal area of disc diagrams bounded by $\alpha_n$ converges to $\infty$.

Thus for WPD we cannot apply \cite[Lemma 2.11]{Martin} of A. Martin to prove that~$\mathcal{C}$ has Strong Concatenation Property. In fact, the first condition of his Strong Concatenation Property fails if we split $\sigma_n$ into paths $\gamma_1,\gamma_2$ from $[x_3]$ to $[x_3+x_1^2]$ yielding $A\geq 2n$.

However, one can recognise variants of the two conditions of his Theorem~1.2 guaranteeing WPD in our Lemmas~\ref{lem:gammaattracting} and~\ref{lem:axisgivesinfo}. In Lemma~\ref{lem:axisgivesinfo} we prove a slightly weaker result than the one needed to apply his Thm~1.2 directly, but in the proof of Theorem~\ref{thm:wpd} we actually establish his condition using Corollary~\ref{cor:geodesic}.
\end{rem}


\end{document}